\documentclass[a4paper,11pt]{amsart}
\usepackage{fancyhdr}
\usepackage{soul}
\usepackage[T1]{fontenc}
\usepackage[english]{babel}
\usepackage[utf8]{inputenc}
\usepackage{amssymb,amsthm,amsmath} 
\usepackage{indentfirst}
\usepackage{color}
\usepackage{braket}
\usepackage[all,arc]{xy}
\usepackage{mathtools}
\usepackage{graphicx}
\usepackage{xcolor}
\usepackage[left=4cm, right=4cm, top=2cm]{geometry}
\usepackage{tikz-cd}
\usepackage[colorlinks]{hyperref}
\usepackage[shadow, color=babyblueeyes!]{todonotes}
\usepackage{booktabs}
\usepackage{xfrac} 
\usepackage{cite}
\usepackage[super]{nth}
\usepackage{cancel}
\hypersetup
{
	colorlinks,
	linkcolor={black!50!black},
	citecolor={black!50!black},
	urlcolor={blue!80!black}
}
\definecolor{babyblueeyes}{rgb}{0.63, 0.79, 0.95}
\newtheorem{theorem}{Theorem}[section]
\newtheorem{lemma}[theorem]{Lemma}
\newtheorem{proposition}[theorem]{Proposition}
\newtheorem{corollary}[theorem]{Corollary}
\theoremstyle{definition}
\newtheorem{definition}[theorem]{Definition}

\theoremstyle{remark}

\newtheorem{remark}[theorem]{Remark}
\newtheorem{example}[theorem]{Example}
\newtheorem{notation}[theorem]{Notation}

\newcommand{\numberset}{\mathbb}
\newcommand{\N}{\numberset{N}}
\newcommand{\R}{\numberset{R}}
\newcommand{\Z}{\numberset{Z}}
\newcommand{\C}{\numberset{C}}
\newcommand{\Q}{\numberset{Q}}
\newcommand{\PP}{\mathbb{P}^\infty\C}
\newcommand{\SSS}{\mathbb{S}} 

\newcommand{\rk}{{{\rm rk}}}

\title[An exposition of GHRR/2 thm in the fancy language of spectra]{An exposition of the topological half of the Grothendieck--Hirzebruch--Riemann--Roch theorem in the fancy language of spectra}

\author{Mattia Coloma}
\address{Università degli Studi di Roma "Tor Vergata"; Dipartimento di Matematica, Via della Ricerca Scientifica, 1 - 00133 - Roma, Italy; 
}
\email{coloma@mat.uniroma2.it}
\author{Domenico Fiorenza}
\address{Sapienza Universit\`a di Roma; Dipartimento di Matematica ``Guido Castelnuovo'', P.le Aldo Moro, 5 - 00185 - Roma, Italy; 
}
\email{fiorenza@mat.uniroma1.it}
\author{Eugenio Landi}
\address{Università di Roma Tre; Dipartimento di Matematica e Fisica Largo San Leonardo Murialdo, 1 - 00146 - Roma, Italy;
}
\email{eugenio.landi@uniroma3.it}
\begin{document}

\renewcommand\refname{Index Librorum Prohibitorum}

\begin{abstract}
We give an informal exposition of pushforwards and orientations in generalized cohomology theories in the language of spectra. The whole note can be seen as an attempt at convincing the reader that Todd classes in Grothendieck--Hirzebruch--Riemann--Roch type formulas are not Devil's appearances but rather that things just go in the most natural possible way. \\
\begin{center}
\includegraphics[scale=0.40]{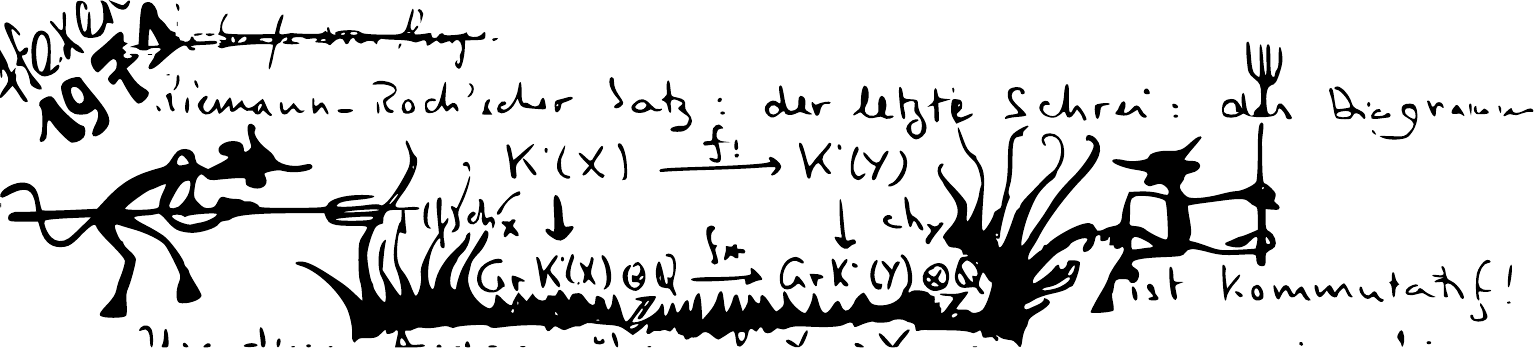}
\end{center}
\end{abstract}
	\maketitle
\tableofcontents

\section{Captatio benevolenti\ae}

After tensoring by $\Q$, the Chern character induces a natural isomorphism of rings
\[
\mathrm{ch}\colon K^0(X)\otimes \Q \to \bigoplus_{i\in \Z}H^{2i}(X;\Q)
\]
from the complex $K$-theory of a (nice) topological space $X$ to its even singular cohomology with rational coefficients. Moreover, both complex $K$-theory and singular cohomology have natural complex orientations. Simplifying a bit, this means that if $X$ is a compact complex manifold, one has pushforward maps
\[
\pi_{X,\ast}^{KU}\colon K^0(X)\otimes \Q\to K^{-2\dim_\C X}(\ast)\otimes \Q=\Q
\]
and
\[
\pi_{X,\ast}^{H\Q}=\int_X\colon \bigoplus_{i\in \Z}H^{2i}(X;\Q)\to \bigoplus_{i\in \Z}H^{2i-2\dim_\C X}(\ast;\Q)=\Q
\]
Since everything is very natural here, one would expect that in the best of possible worlds the diagram
\[
\begin{tikzcd}
K^0(X)\otimes \Q \arrow[dr, "\pi_{X,\ast}^{KU}"'] \arrow[rr, "\mathrm{ch}", "\sim"']      &                                                & \bigoplus_{i\in \Z}H^{2i}(X;\Q) \arrow[dl, "\int_X"]                 \\
 & \Q &  
\end{tikzcd}
\]
would commute, i.e., to have an integral formula of the form
\[
\pi_{X,\ast}^{KU}([V])=\int_X \mathrm{ch}(V)
\]
for any complex vector bundle $V$ on $X$.
Unfortunately, this formula is notoriously not correct: it becomes so only after the introduction of a suitable multiplicative correction factor, the Todd class of $X$, given by the cohomology class
\[
\mathrm{td}(X)=
1+ \frac{c_1(X)}{2}+\frac{c_1(X)^2+c_2(X)}{12}+\frac{c_1(X)c_2(X)}{24}+\cdots
\]
This rather formidable expression is obtained from the characteristic power series
\[
\frac{u}{1-e^{-u}}=1+\frac{u}{2}+\frac{u^2}{12}-\frac{u^4}{720}+\cdots
\]
by applying the splitting principle to the holomorphic tangent bundle $TX$ of $X$, i.e., if $TX$ splits as a direct sum of complex line bundles $L_i$, then
\[
\mathrm{td}(X)=\prod_{i=1}^{\dim_\C X}\frac{u}{1-e^{-u}}\biggr\vert_{u=c_1(L_i)}.
\]
On first sight, both the presence of the factor $\mathrm{td}(X)$ in the corrected formula
\[
\pi_{X,\ast}^{KU}([V])=\int_X \mathrm{ch}(V)\,\mathrm{td}(X)
\]
and its expression in terms of Chern classes may appear rather mysterious. The main goal of this paper is to try to convince the reader that there is actually no mystery here, and that on the contrary the specific correction factor $\mathrm{td}(X)$ is precisely what one should have expected from the very beginning. We will try to achieve this within an informal exposition of the theory of pushforwards and orientations in cohomology, 
with an emphasis on categorical features of the category of spectra coming into play. 
\\
We make no claim of originality. 
Everything we write can be found elsewhere, and plenty of references will be given throughout the text. We owe most of our gratitude for the inspiration to Ando, Blumberg and Gepner and their paper \cite{ABG}, to Panin and Smirnov for their \cite{paninsmirnov} and \cite{panin2002riemann}, and of course to Quillen \cite{quillen1971elementary}.
We have also blatantly stolen from Lurie's \cite{Lurieelliptic} and \cite{luriechromatic}, from Ando, Blumberg Gepner, Hopkins and Rezk's \cite{ando2008units} and \cite{ando2010multiplicative}, and from May's classic \cite{May-E-infty}.\\
We wish to reassure the reader that, although at the very beginning we will be mentioning that the category of spectra is a symmetric monoidal stable $\infty$-category in the sense Lurie's Higher Algebra \cite{lurie2012higher}, no advanced knowledge of $\infty$-category theory is actually required to read this note as we will treat the higher categorical aspects of spectra only very intuitively and we will be mostly working in the homotopy category of spectra, which is an ordinary category (i.e., a 1-category).\\
We thank Bertram Arnold, Nicholas Kuhn, Peter May, Denis Nardin and Dylan Wilson, whose answers to our random and often naive questions on MO have actually been an invaluable encouragement at an early stage of this research; Matthias Kreck and Peter Teichner, for influential comments on a preliminary draft of these notes and on a talk based on them; Urs Schreiber, the Referee and the Editor for comments, corrections and suggestions that greatly helped us in improving the exposition from the first arXiv version to the present one; Fosco Loregian, for countless coffees, beers and profunctors shared with us; and Cliff Booth, for having pushed us to conclude the writing of this note. d.f. thanks MPIM Bonn for the excellent and friendly research environment that surrounded a revision of this note in a rainy week of November. 

\section{Introitus}
We will be working in the stable $\infty$-category $\mathsf{Sp}$ of spectra. 
We reassure the reader possibly unfamiliar both with spectra and/or with $\infty$-categories (stable or not) that a previous knowledge of these may be useful but not necessary at all. A good motto to keep in mind is the following informal analogy: spectra are to spaces as real numbers are to rational numbers. By this we do not only mean that spectra are certain sequences of (nice) topological spaces just like real numbers are (equivalence classes of) certain sequences of rational numbers, but also --and this is the main content of the motto-- that one actually works with spectra by knowing the categorical features of the category they form and its relations to the category of topological spaces, more than with their actual definition entailing these features. In terms of the analogy, this is to say that $\mathbb{R}$ can be defined as a suitable quotient of the set of Cauchy sequences of rational numbers, but what one generally uses when working with real numbers is not this definition but the properties it implies: e.g., that $\mathbb{R}$ is a complete ordered field. This is actually a complete definition of $\mathbb{R}$: a complete   
ordered field, if it exists, is unique and contains $\mathbb{Q}$. From this point of view, the whole business of equivalence classes of Cauchy sequences can be seen as a proof of the existence of a field with the completeness and ordering properties. Turning back to spectra, we are saying that the only point we are asking the reader to trust us on is that there exists a category $\mathsf{Sp}$ having the properties we will attribute to it.
\par
Actually, $\mathsf{Sp}$ will not be an ordinary category, but an $(\infty,1)$-category. We address the interested reader to \cite{Lurie-htt} for a comprehensive treatment, and here content us with saying that for two given objects $X$ and $Y$ in $\mathsf{Sp}$, we don't have just a \emph{set} of morphisms, but a whole \emph{space} of morphisms $\mathsf{Sp}(X,Y)$ between $X$ and $Y$. This allows us to say not only that two morphisms are possibly equal, but that they are possibly homotopic (i.e., they lie in the same connected component of $\mathsf{Sp}(X,Y)$), or to talk of homotopies between homotopies between morphisms, and so on. Also, commutative diagrams of spectra will not be strictly commutative, but always commutative up to some given homotopy, which is part of the data defining a commutative diagram in an $\infty$-category. 

\begin{remark}
Calling the objects of $\mathsf{Sp}$ ``$0$-morphisms'', the elements in $\mathsf{Sp}(X,Y)$ ``$1$-morphisms, the homotopies between 1-morphisms ``$2$-morphisms'' and so on, one sees that in $\mathsf{Sp}$ one has $k$-morphisms for every $k\geq 0$. This motivates the terminology $\infty$-category. Moreover, since every homotopy is invertible up to a homotopy between homotopies, one sees that for $k>1$ every $k$-morphism in $\mathsf{Sp}$ is invertible up to $(k+1)$-morphisms. One indicates that 1 is the threshold for invertibility by saying that the $\infty$-category $\mathsf{Sp}$ is an $(\infty,1)$-category.
\end{remark}

Taking connected components of the hom-spaces one gets the homotopy category of spectra, usually denoted by $h\mathsf{Sp}$. As a matter of notation, hom-sets in the category 
$h\mathsf{Sp}$ will be denoted by $[-,-]$ rather than by $h\mathsf{Sp}(-,-)$, that is one writes
\[
[X,Y]:=\pi_0\mathsf{Sp}(X,Y),
\]
for any $X,Y$ in $\mathsf{Sp}$. For a morphism $f\colon X\to Y$ in $\mathsf{Sp}$, we will write
\[
f^*\colon \mathsf{Sp}(Y,Z)\to \mathsf{Sp}(X,Z)
\]
and
\[
f_*\colon \mathsf{Sp}(T,X)\to \mathsf{Sp}(T,Y)
\]
for the continuous maps between hom-spaces induced by precompostion and postcomposition with $f$, respectively. We will use the same symbols to denote the induced maps between hom-sets in $h\mathsf{Sp}$.
\begin{remark}
Even by the few informal lines above, one should deduce that $\infty$-categories are a nice context for doing homotopical constructions. In classical category theory, such a context is provided by the notion of a model structure on a category. Starting with a model structure on an ordinary category is indeed one of the most powerful ways to produce a rigorously defined $(\infty,1)$-category. In these cases one says that the $(\infty,1)$-category is \emph{presented} by the model structure. For instance, a rigorous definition of the $(\infty,1)$-category of spectra we are talking about is as the $(\infty,1)$-category presented by the standard model structure on the (ordinary) category of orthogonal spectra \cite{mmss}.
\end{remark}

Other two $(\infty,1)$-categories we will meet in this note are the $\infty$-category $\mathsf{Top}$ of (nice) topological spaces, presented by the classical (or Quillen) model structure, and whose homotopy category $h\mathsf{Top}$ is the usual homotopy category of (nice) topological spaces, and its pointed version $\mathsf{Top}_\ast$. The $\infty$-categories $\mathsf{Top}$, $\mathsf{Top}_\ast$ and $\mathsf{Sp}$ are related by the 
 functors (or, better, $\infty$-functors)
\[
	\begin{tikzcd}
\mathsf{Top} \arrow[r, "(-)_+"] & \mathsf{Top}_* 
\end{tikzcd},
\]
and
\[
\begin{tikzcd}
\mathsf{Top}_* \arrow[r, "\Sigma^\infty"] & \mathsf{Sp}
\end{tikzcd},
\]
where $(-)_+\colon \mathsf{Top}\to \mathsf{Top}_*$  is the functor adjoining a free basepoint, i.e., $X_+:=X\sqcup \ast$, and $\Sigma^\infty\colon \mathsf{Top}_*\to \mathsf{Sp}$ is the infinite suspension, mapping a pointed space $Y$ to the spectrum given by sequence of pointed spaces $\Sigma^\infty Y$ recursively defined by $(\Sigma^\infty Y)_0=Y$ and $(\Sigma^\infty Y)_{n+1}=\Sigma((\Sigma^\infty Y)_n)$ for any $n>0$.
The morphisms $\Sigma((\Sigma^\infty Y)_n)\to (\Sigma^\infty Y)_{n+1}$ giving $\Sigma^\infty Y$ the structure of a (sequential) spectrum are the identities.
\par
We will always look at (nice) topological spaces and pointed (nice) topological spaces as spectra via the sequence of functors
\[
	\begin{tikzcd}
\mathsf{Top} \arrow[r, "(-)_+"] & \mathsf{Top}_* \arrow[r, "\Sigma^\infty"] & \mathsf{Sp}
\end{tikzcd},
\]
that is, given a space $X$, we will denote by the same symbol $X$ its stabilization, i.e., the spectrum  $\Sigma^\infty_+X$, and similarly for pointed spaces. As a notable exception to this rule, when $X$ is the space $\ast$ consisting of a single point, we will use the classical notation $\SSS$ to denote the spectrum $\Sigma^\infty_+\ast$. 

It should be remarked that stabilization is not an embedding: by stabilising one loses all the unstable information about a space $X$. Nevertheless, as we will only be concerned with stable aspects, no confusion should arise in this note, and we'll find it an extremely convenient notation to denote by the same symbol both a space and its stabilization.
\begin{remark}
By definition, the spectrum $\mathbb{S}$ has $\mathbb{S}_0=\ast_+=\ast\sqcup \ast=S^0$ and so, inductively, $\mathbb{S}_n=S^n$ for any $n\geq 0$. It is called the \emph{sphere spectrum}.
\end{remark}

\begin{remark}
In an $(\infty,1)$-category, all universal properties are given ``up to homotopy'' (with the homotopies being part of the universal property). For instance, in an ordinary category $\mathsf{C}$ the property of an object $\emptyset$ of being initial is the fact that the hom-set $\mathsf{C}(\emptyset,X)$ is a singleton for every object $X$ in $\mathsf{C}$. In the $(\infty,1)$-categorical setting this is translated into requiring that the hom-space $\mathsf{C}(\emptyset,X)$ is ``a singleton up to homotopy'', i.e., into the requirement that $\mathsf{C}(\emptyset,X)$ is contractible. This means that there is not a unique initial morphism, but that the initial morphism is unique up to homotopies that are in turn unique up to higher homotopies, and so on. With this in mind, one continues to talk of ``the'' initial morphism $\emptyset \to X$. Similarly, the universal property of the pullback, declined into an $(\infty,1)$-categorical setting, becomes the fact that the commutative diagram
\[
\begin{tikzcd}
    \mathsf{C}(T,X\times_Z Y)\arrow[d]\arrow[r]&\mathsf{C}(T,X)\arrow[d]\\
    \mathsf{C}(T,Y)\arrow[r]&\mathsf{C}(T,Z)
    \end{tikzcd}
\]
is a \emph{homotopy pullback} of topological spaces for every $T$ in $\mathsf{C}$. We refer the reader to \cite{doeraene} for an introduction to homotopical pullbacks and pushouts in the category of (nice) topological spaces. 
\end{remark}

By saying that the $\infty$-category of spectra is stable one means it satisfies two very simple axioms (see \cite{lurie2012higher} for a comprehensive account):
\begin{enumerate}
    \item it has a zero object $\mathbf{0}$, i.e., an object that is at the same time initial and terminal;
    \item it has pullbacks and pushouts, and every pullback is a pushout, and vice versa.
\end{enumerate}
\begin{remark}
Since pullback and pushout diagrams coincide in a stable $(\infty,1)$-category, one sometimes calls them ``pullout'' diagrams. Pullout diagrams of the form
\[
\begin{tikzcd}
    X\arrow[d]\arrow[r,"f"]&Y\arrow[d,"g"]\\
    \mathbf{0}\arrow[r]&Z
    \end{tikzcd}
\]
are called fiber/cofiber sequences. In this case, one says that $X$ is the fiber of $g$ and $Z$ is the cofiber of $g$.
\end{remark}
\begin{remark}
One uses a special notation to denote the fiber of the initial morphism $\mathbf{0}\to X$ and the cofiber of the terminal morphism $X\to \mathbf{0}$. Namely one denotes them by the symbols $X[-1]$ and $X[1]$, respectively, so that one has the following defining pullout diagrams for these:
\[
\begin{tikzcd}
    X[-1]\arrow[d]\arrow[r]&\mathbf{0}\arrow[d]\\
    \mathbf{0}\arrow[r]&X
    \end{tikzcd}\quad;\qquad
\begin{tikzcd}
    X\arrow[d]\arrow[r]&\mathbf{0}\arrow[d]\\
    \mathbf{0}\arrow[r]&X[1]
    \end{tikzcd}\,.
\]
Comparison of the two pullout diagrams above immediately shows that in a stable $(\infty,1)$-category $\mathsf{C}$ the shifts
\[
[1],\, [-1]\colon \mathsf{C}\to \mathsf{C}
\]
are autoequivalences of $\mathsf{C}$ inverse to each other. One calls $[1]$ the \emph{shift} functor and, for any $n\in \mathbb{Z}$ one writes $[n]$ for $[1]^n$.
\end{remark}
\begin{remark}
The reader used to the language of triangulated categories will have found this sketchy description of stable $\infty$-categories echoing something familiar. And indeed the homotopy category $h\mathsf{C}$ of a stable $(\infty,1)$-category carries a natural structure of a triangulated category, where distinguished triangles are the image in $h\mathsf{C}$ of fiber/cofiber sequences in $\mathsf{C}$ and where the shift functor is induced by the shift functor of $\mathsf{C}$. It is a nice exercise to show that the two simple axioms of stable $(\infty,1)$-category imply the somehow less transparent ``octahedral axiom'' of triangulated categories.
\end{remark}
\begin{remark}
The shift functor in the $(\infty,1)$-category of spectra is induced by the suspension functor on pointed topological spaces. By this reason it is commonly denoted as $X\mapsto \Sigma X$ in algebraic topology textbooks. Here we prefer denoting it by $X\mapsto X[1]$ to stress that it is the shift functor. Since $[1]$ and $[-1]$ are inverse autoequivalences of $\mathsf{C}$, one has a natural homotopy equivalence $\mathsf{Sp}(X[1],Y)\cong\mathsf{Sp}(X,Y[-1])$ for any $X,Y$ in $\mathsf{Sp}$. The identification of $[1]$ with the suspension functor then identifies the negative shift $Y\mapsto Y[-1]$ with the loop space functor.
\end{remark}

\begin{remark} 
As in every stable $\infty$-category, the hom-sets $[X,Y]$ in the homotopy category $h\mathsf{Sp}$ have a natural structure of  abelian groups. Namely, since $\mathbf{0}$ is the zero object, we have natural homotopy equivalences $\mathsf{Sp}(X,\mathbf{0})\cong \ast$ and so, by definition of the negative shift functor, the hom-space $\mathsf{Sp}(X,Y[-1])$ is defined by the homotopy pullback
\[
\begin{tikzcd}
    \mathsf{Sp}(X,Y[-1])\arrow[d]\arrow[r]&\mathbf{0}\arrow[d]\\
    \mathbf{0}\arrow[r]&\mathsf{Sp}(X,Y)
    \end{tikzcd}.
\]
This gives a natural identification of $\mathsf{Sp}(X,Y[-1])$ with the loop space $\Omega\mathsf{Sp}(X,Y)$ based at the zero morphism $X\to Y$, i.e., at the morphism $X\to \mathbf{0}\to Y$. Iterating this, one gets an identification $\mathsf{Sp}(X,Y[-n])\cong \Omega^n\mathsf{Sp}(X,Y)$ for every $n\geq 1$. In particular, by using $Y\cong Y[2][-2]$, this gives the the natural identification
\[
[X,Y]=\pi_0\mathsf{Sp}(X,Y)=\pi_0\Omega^2\mathsf{Sp}(X,Y[2])=\pi_2\mathsf{Sp}(X,Y[2]).
\]
\end{remark}

\begin{definition}
Let $X$ be a spectrum. For any $n\in\mathbb{Z}$, one writes 
\[
\pi_n^{\mathsf{Sp}}(X):=[\mathbb{S}[n],X]
\]
and calls the abelian group $\pi_n^{\mathsf{Sp}}(X)$ the $n$-th homotopy group of $X$.
\end{definition}
\begin{remark}
Notice that if $X$ is a pointed space seen as a spectrum, then for $n\geq 0$, the homotopy group $\pi_n^{\mathsf{Sp}}(X)$ is \emph{not} the $n$-th homotopy group of $X$: it is the $n$-th \emph{stable} homotopy group of $X$.
\end{remark}

Another remarkable feature of the $\infty$-category of spectra is that it is symmetric monoidal with the smash product as tensor product and the sphere spectrum $\mathbb{S}$ as unit object. Following \cite{luriechromatic} we will denote the smash product of spectra with the symbol $\otimes$ rather than with the usual $\wedge$ to stress it is the monoidal product.  Similarly, we will use $\oplus$ instead of $\vee$ to denote the coproduct of spectra. Also $\mathsf{Top}$ and $\mathsf{Top}_*$ are symmetric monoidal, with tensor product given by the Cartesian product and by the smash product, respectively. Moreover,  both $(-)_+\colon \mathsf{Top}\to \mathsf{Top}_*$ and $\Sigma^\infty\colon \mathsf{Top}_*\to \mathsf{Sp}$ are monoidal functors \cite{mmss}. This fact has an immediate important consequence: spaces are special objects within spectra with respect to the monoidal structure. Namely, every object in $\mathsf{Top}$ is a coassociative cocommutative comonoid via the diagonal morphism $\Delta\colon X\to X\times X$ and the terminal morphism $X\to \ast$. Since $\Sigma^\infty_+$ is symmetric monoidal we have that a space $X$, seen as a spectrum, comes with natural 
 distinguished morphisms
\begin{align*}
 \Delta\colon  X&\to X\otimes X\\  
  \epsilon\colon X&\to \mathbb{S}
\end{align*}
making it a coassociative cocommutative comonoid in $\mathsf{Sp}$.

\begin{remark}
Tensor product of spectra is compatible with the shift functor: one has natural isomorphisms
\[
X\otimes (Y[1])\cong (X\otimes Y)[1] \cong (X[1])\otimes Y.
\]
These identify the shift functor with the tensor product with the object $\mathbb{S}[1]$.
\end{remark}

\begin{definition}
Algebra objects (or monoids) in $\mathsf{Sp}$, i.e., spectra $E$ endowed with morphisms 
\begin{align*}
  m\colon &E\otimes  E \to E\\  
  e\colon &\mathbb{S}\to E
\end{align*}
satisfying the usual unit and associativity conditions will be called \emph{ring spectra}. If, moreover, the multiplication is commutative (up to coherent homotopies) they will be called \emph{commutative ring spectra} or $E_\infty$-ring spectra.
\end{definition}
\begin{remark}
Since shift commutes with tensor product, if $E$ is a ring spectrum then $E[k]$ is an $E$-module for any $k\in \mathbb{Z}$.
\end{remark}

As in every monoidal category, if $X$ is a comonoid and $E$ is a monoid, then $[X,E]$ is a monoid with multiplication given by the composition
\[
\begin{tikzcd}
 {[X,E] \otimes [X,E]} \arrow[r, "\otimes"] & {[X\otimes X, E\otimes E]} \arrow[rr, "{(\Delta^*,m_*)}"] && {[X,E]}
\end{tikzcd}
\]
As this multiplication is compatible with the abelian group structure on $[X,E]$, when $X$ is a comonoid and $E$ is a monoid the hom-set $[X,E]$ has a canonical ring structure, which is commutative if both $E$ is commutative and $X$ is cocommutative. So, in particular, if $X$ is a space and $E$ is an $E_\infty$-ring spectrum, then $[X,E]$ is a commutative ring.

\begin{remark}
Thanks to the compatibility of the shift functor with the tensor product, if $X$ is a space then the direct sum
\[
\bigoplus_{n\in \mathbb{Z}}[X,E[n]]
\]
has a natural structure of graded commutative ring. This is called the graded $E$-cohomology ring of $X$.
\end{remark}

\begin{remark}\label{rem:push-and-pull-in-cohomology}
If $\psi\colon E\to F$ is a morphism of $E_\infty$-ring spectra, then 
\[
\psi_*\colon [X,E]\to [X,F]
\]
is a homomorphism of rings, for any comonoid $X$. Dually, if $\phi\colon X\to Y$ is a morphism of comonoids, then
\[
\phi^*\colon [Y,E]\to [X,E]
\]
is a homomorphism of rings, for any  $E_\infty$-ring spectrum $E$. In particular, since any continuous map of (nice) topological spaces is a comonoid map with respect to the comultiplication given by the diagonal embedding, any continuous map between spaces induces a pullback ring homomorphism $\phi^*\colon [Y,E]\to [X,E]$. As a special case of this, by taking the terminal morphism we see that for any space $X$ the ring $[X,E]$ comes with a natural ring homomorphism 
\[
\pi_0^{\mathsf{Sp}}(E)=[\mathbb{S},E]\to [X,E],
\]
making $[X,E]$ a $\pi_0^{\mathsf{Sp}}(E)$-algebra. One says that $\pi_0^{\mathsf{Sp}}(E)$ is the \emph{ring of coefficients} for the multiplicative cohomology theory defined by the $E_\infty$-ring spectrum $E$.
\end{remark}

\begin{example}
Let $A$ be an abelian ring. If $E=HA$, the Eilenberg--Mac Lane $E_\infty$-ring spectrum defining singular cohomology with coefficients in $A$, one has $\pi_0^{\mathsf{Sp}}(HA)=A$, so that the coefficients of singular cohomology are indeed the coefficients in the sense of Remark \ref{rem:push-and-pull-in-cohomology}.
\end{example}

\begin{remark}
Since shift commutes with tensor product, 
if $X$ is a comonoid in $\mathsf{Sp}$ and $Y$ is a comodule over $X$, i.e., we have a morphism
\[
\rho\colon Y\to X\otimes Y
\]
making the obvious diagrams commute, then for any $n\in \mathbb{Z}$ also $Y[n]$ is a $X$-comodule. In particular, for any $n$ the hom-set $[Y[n],E]$ is an $[X,E]$-module via the $[X,E]$-action
\[
\begin{tikzcd}
 {[X,E] \otimes [Y[n],E]} \arrow[r, "\otimes"] & {[X\otimes Y[n], E\otimes E]} \arrow[rr, "{(\rho[n]^*,m_*)}"] && {[Y[n],E]}.
\end{tikzcd}
\]
A particular instance of this construction is obtained by taking $X=\mathbb{S}$ and $\rho\colon Y\xrightarrow{\sim} \mathbb{S}\otimes Y$ the natural isomorphism, for any spectrum $Y$. This way, one sees that $[Y[n],E]$ is a $[\mathbb{S},E]$-module, i.e., a $\pi_0^{\mathsf{Sp}}(E)$-module, for any $n\in \mathbb{Z}$. In particular, if $Y$ is a space, this module structure coincides with the one induced by the ring homomorphism $\pi_0^{\mathsf{Sp}}(E)\to [Y,E]$ described in Remark \ref{rem:push-and-pull-in-cohomology}.
\end{remark}

\begin{remark}
If $X$ and $Y$ are two (nice) topological spaces and $E$ is an $E_\infty$-ring spectrum, then the commutative diagram of spaces
\[
\begin{tikzcd}
 {X\times Y} \arrow[r]\arrow[d] & {X} \arrow[d] \\
 {Y}\arrow[r] & {\ast}
\end{tikzcd}
\]
induces the commutative diagram of commutative rings
\[
\begin{tikzcd}
 {[X\otimes Y,E]}  & {[X,E]} \arrow[l] \\
 {[Y,E]}\arrow[u] & {[\SSS,E]} \arrow[l]\arrow[u]
\end{tikzcd}
\]
and so, by the universal property of the tensor product of commutative rings, a morphism of rings
\[
\otimes_E\colon [X,E]\otimes_{\pi_0^{\mathsf{Sp}}(E)}[Y,E]\to [X\otimes Y,E].
\]
One sees that $\otimes_E$ is induced by the external multiplication
\[
\begin{tikzcd}
 {[X,E] \otimes [Y,E]} \arrow[r, "\otimes"] & {[X\otimes Y, E\otimes E]} \arrow[rr, "{m_*}"] && {[X\otimes Y,E]},
\end{tikzcd}
\]
by factoring the latter through $[X,E] \otimes [Y,E]\to [X,E] \otimes_{\pi_0^{\mathsf{Sp}}(E)} [Y,E]$.
\end{remark}

\begin{definition}
We will say that the $E_\infty$-ring spectrum $E$ is \emph{periodic of period $k$} or simply \emph{k-periodic} if we are given an isomorphism of $E$-modules $\phi:E\to E[-k]$. The composition
\[
\beta_E\colon \mathbb{S}\to E\xrightarrow{\phi} E[-k]
\]
is called the \emph{Bott element} of the $k$-periodic $E_\infty$-ring spectrum $(E,\phi)$.
\end{definition}
\begin{remark}
At the level of homotopy classes, the Bott element $\beta_E$ is an element in $\pi^{\mathsf{Sp}}_k(E)$. 
\end{remark}
\begin{remark}
Since $\phi$ is an isomorphism, the Bott element is invertible: there exists an element $\beta_E^{-1}\colon \SSS\to E[k]$ with $\beta_E\beta_E^{-1}=\beta_E^{-1}\beta_E=e\colon \SSS\to E$. For any spectrum $X$, the multiplication by the Bott element induces a natural isomorphism of abelian groups $[X,E]\xrightarrow{\beta_E\cdot} [X[k],E]$. In particular, the choice of an element $\xi \in [X[-k],E]$ is equivalent to the choice of an element $\beta_E\xi \in [X,E]$ via the periodicity isomorphisms. If $X$ is a space, by the universal property of polynomial rings the latter are in bijection with ring homomorphisms $\pi^{\mathsf{Sp}}_0(E)[u] \to [X,E]$ extending the canonical homomorphism $\epsilon^*\colon [\mathbb{S},E]\to [X,E]$, via the association $u \mapsto \beta_E\xi$.
\end{remark}
\begin{remark}
If $(E,\phi)$ is a $k$-periodic $E_\infty$-ring spectrum, then the isomorphism $\phi$ is the multiplication by the Bott element, i.e., $\phi$ is the composition
\[
\begin{tikzcd}
{E\cong E\otimes \SSS } \arrow[r, "{\mathrm{id}\otimes \beta}"]& {E\otimes E[k]\cong (E\otimes E)[-k]}\arrow[r, "{m[-k]}"] & {E[-k]}.
\end{tikzcd}
\]
Therefore, one can equivalently define a $k$-periodic $E_\infty$-ring spectrum as an $E_\infty$-ring spectrum endowed with a morphism $\beta_E\colon \mathbb{S}[k]\to E$ such that multiplication by $\beta_E$ induces an isomorphism of $E$-modules $E\to E[-k]$.
\end{remark}

\begin{definition} An $E_\infty$-ring spectrum
$E$ is called \emph{even $2$-periodic} if it is $2$-periodic and $\pi_{2k+1}^{\mathsf{Sp}}(E) = 0, \forall k \in \Z$.
\end{definition}

\begin{example}\label{example:chern-bott}
The first and most natural example of an even $2$-periodic cohomology theory is complex $K$-theory $KU$. The Bott element $\beta_{KU}\in\pi_2^{\mathsf{Sp}}(KU)=K^0(S^2)$ can be identified with (the class of the) virtual line bundle $\mathbf{1}_\C-L^{-1}$ in the $K$-theory of $S^2$, where $L=\mathcal{O}_{\mathbb{P}^1\C}(1)$
is the universal complex line bundle over the complex projective line $\mathbb{P}^1\C\cong S^2$ and $\mathbf{1}_\C$ is the trivial rank 1 complex vector bundle $S^2\times \C\to S^2$.

Another even $2$-periodic cohomology theory we will come back to later is even 2-periodic rational singular cohomology $HP_{\mathrm{ev}}\Q:=\bigoplus_{i\in\Z}H\Q[2i]$. Here, as usual, $H\Q$ denotes the Eilenberg-Mac Lane spectrum of the abelian group $(\Q,+)$, so that $[X,H\Q[2i]]=H^{2i}(X;\Q)$ for any (nice) topological space $X$. Its homotopy groups are given by
\[
\pi_i^{\mathsf{Sp}}(H\Q)=\begin{cases}
\Q \text{ if } i=0\\
0 \text{ otherwise.}
\end{cases}
\]
The Bott element $\beta_{HP_{\mathrm{ev}}\Q}$ of $HP_{\mathrm{ev}}\Q$ can be identified with $1\in \pi_0^{\mathsf{Sp}}(H\Q)$ via 
\[
\pi_2^{\mathsf{Sp}}(HP_{\mathrm{ev}}\Q)=\prod_{i\in\Z}\pi_{2-2i}^{\mathsf{Sp}}(H\Q)\beta^i=\pi_0^{\mathsf{Sp}}(H\Q)\beta=\Q\beta,
\]
where $\beta$ is a degree $-2$ formal variable used to keep track of the degree shiftings. In other words, the Bott element $\beta_{HP_{\mathrm{ev}}\Q}$ of $HP_{\mathrm{ev}}\Q$ is naturally identified with the formal variable $\beta$ used in forming the even 2-periodic rational singular cohomology of a space $X$:
\[
HP^k_{\mathrm{ev}}(X;\mathbb{Q})=\bigoplus_{i\in \mathbb{Z}}H^{k+2i}(X;\Q)\beta^i.
\]
Equivalently, the Bott element $\beta_{HP_{\mathrm{ev}}\Q}$ is the fundamental class of $S^2$ via
\[
\pi_2^{\mathsf{Sp}}(HP_{\mathrm{ev}}\Q)=\prod_{i\in\Z}\pi_{2}^{\mathsf{Sp}}(H\Q[2i])=\pi_2^{\mathsf{Sp}}(H\Q[2])=\Q\beta.
\]
Notice that at the level of $\pi_2^{\mathsf{Sp}}$'s, the Chern character $\mathrm{ch}\colon KU \to HP_{\mathrm{ev}}\Q$ gives
$\pi_2^{\mathsf{Sp}}(\mathrm{ch})\colon \pi_2^{\mathsf{Sp}}(KU)\to \pi_2^{\mathsf{Sp}}(HP_{\mathrm{ev}}\Q)$ mapping $\mathbf{1}_\C-L^{-1}$ to $1-e^{\beta c_1(L^{-1})}=\beta c_1(L)$. As the first Chern class of the universal line bundle on $\mathbb{P}^1\C$ represents the fundamental class of $\mathbb{P}^1\C$ in singular cohomology (with $\mathbb{Z}$- and so also) with $\mathbb{Q}$-coefficients, we see that the Chern character maps the Bott element of $KU$ to the Bott element of $HP_{\mathrm{ev}}\Q$.
\end{example}

We will also be prominently considering the category $\mathsf{VectBun}_{\mathbb{R}}$ of real vector bundles over (nice) topological spaces, whose objects are vector bundles $V\to X$ and whose morphisms are commutative diagrams
\[
\begin{tikzcd}
V \arrow[r, "{\tilde{f}}"] \arrow[d] & W \arrow[d] \\
X \arrow[r, "f"]             & Y                  
\end{tikzcd}
\]
where $\tilde{f}$ is fibrewise linear. The category $\mathsf{VectBun}_{\mathbb{R}}$ is symmetric monoidal with monoidal product given by $(V\to X)\otimes (W\to Y)=(V\boxplus W\to (X\times Y))$, where $(V\boxplus W)_{(x,y)}=V_x\oplus W_y$, and unit object the zero vector bundle over the point. 
\begin{notation}
With $K$-theory in mind, in what follows we will usually denote the direct sum of vector spaces and of vector bundles by $+$ instead than by $\oplus$.
\end{notation}

\begin{definition}
The \emph{Thom space} $\mathrm{Th}(V\to X)$ of a real vector bundle $V\to X$ is defined as the homotopy pushout
\[
\begin{tikzcd}
V\setminus X \arrow[r] \arrow[d] & V \arrow[d] \\
\ast \arrow[r, "f"]             & \mathrm{Th}(V\to X)   \end{tikzcd}
\]
where $V\setminus X$ denotes the total space of the vector bundle $V$ minus the copy of $X$ inside it given by the zero section.
\end{definition}
Notice that, by its very definition, $\mathrm{Th}(V\to X)$ is a pointed space. The basepoint of $\mathrm{Th}(V\to X)$ is customary denoted by $\infty$: it is the ``point at infinity'' of the total space of the vector bundle $V\to X$ in the ``vertical directions''.
\begin{remark}
A more concrete description of the Thom space $\mathrm{Th}(V\to X)$ is obtained as follows. In order to compute the homotopy pushout defining it, one replaces the inclusion $V\setminus X \hookrightarrow V$ with the homotopy equivalent cofibration $V\setminus B(V)\hookrightarrow V$, where $B(V)$ is the open unit disk bundle of $V$ for some chosen Riemannian metric. By retracting on the closed unit disk bundle, this is in turn equivalent to the cofibration  $S(V)\hookrightarrow D(V)$, where $S(V)$ denotes the unit sphere bundle and $D(V)$ the  closed unit disk bundle of $V$, respectively. The Thom space of $V$ is then realized as the quotient space $D(V)/S(V)$, with base point the equivalence class $S(V)$. One sees from this description that when the base space $X$ of the vector bundle $V\to X$ is compact, the Thom space $\mathrm{Th}(V\to X)$ reduces to the one-point compactification of the total space $V$, pointed at the point at infinity.
\end{remark}
\begin{remark}
The Thom space construction,
\[
\mathrm{Th}\colon {\mathsf{VectBun}_{\mathbb{R}}}\to \mathsf{Top}_*
\]
is a symmetric monoidal functor: one has a natural isomorphism
\[
\mathrm{Th}(V\boxplus W\to (X\times Y))\cong \mathrm{Th}(V\to X)\wedge \mathrm{Th}(W\to Y).
\]
\end{remark}
The above remark immediately leads to the following
\begin{definition}
The \emph{Thom spectrum} functor $(V\to X)\mapsto X^V$ is the symmetric monoidal functor ${\mathsf{VectBun}_{\mathbb{R}}} \to \mathsf{Sp}$
given by the composition
\[
	\begin{tikzcd}
{\mathsf{VectBun}_{\mathbb{R}}} \arrow[r, "\mathrm{Th}"] & \mathsf{Top}_* \arrow[r, "\Sigma^\infty"] & \mathsf{Sp}
\end{tikzcd}.
\]
\end{definition}
\begin{example}\label{ex:ast-R}
If $V$ is the rank zero vector bundle $\mathbf{0}=X\times \mathbb{R}^0\to X$, then $V\setminus X=\emptyset$ and so $\mathrm{Th}(\mathbf{0}\to X)=X_+$. This implies $X^{\mathbf{0}}=X$, where as usual on the right hand side we are writing $X$ for the suspension spectrum $\Sigma^\infty_+X$. Another simple example is the following.
As $\mathrm{Th}(\mathbb{R}\to \ast)=(S^1,\infty)=\Sigma_+(\ast)$, we have $\ast^\mathbb{R}=\SSS[1]$.
\end{example}

\begin{remark}\label{rem:positive-shift}
In the categorical spirit of this note, the actual definition of the Thom spectrum functor is less important than its properties. What the reader should keep in mind is that to a real vector bundle $V\to X$ is associated a spectrum $X^V$ in such a way that $(X\times Y)^{V\boxplus W}\cong X^V\otimes Y^W$, and that $\ast^\mathbb{R}=\SSS[1]$. For instance, from these two properties one derives
\[
X^{V+ \mathbf{1}}=(X\times\{\ast\})^{V\boxplus \mathbf{1}}=X^V\otimes \SSS[1]=X^V[1],
\]
where we have written $\mathbf{1}$ for the trivial rank 1 vector bundle $\mathbf{1}:=X\times \mathbb{R}\to X$. 
More generally, writing $\mathbf{n}:=X\times \mathbb{R}^n\to X$ for the trivial rank $n$ bundle, one has $X^{V+ \mathbf{n}}=X^V[n]$ for any nonnegative $n$. 
\end{remark}

The other categorical property of Thom spectra that we will need is that
Thom spectra are comodules over their bases: for every (nice) topological space $X$ and every vector bundle $V\to X$ the Thom spectrum $X^V$ has a canonical $X$-comodule structure
\[
\Delta_V\colon X^V \to X\otimes X^V.
\]
This is simply obtained by applying the Thom spectrum functor to the morphism in ${\mathsf{VectBun}_{\mathbb{R}}}$
given by the commutative diagram
\[
\begin{tikzcd}
V \arrow[r, "{(0,\mathrm{id}_V)}"] \arrow[d] & \mathbf{0}\boxplus V \arrow[d] \\
X \arrow[r, "\Delta"]             & X\times X       \end{tikzcd},
\]
where $\Delta\colon X\to X\times X$ is the diagonal embedding.
This construction is natural: if $(f,\tilde{f})\colon (V\to X) \to (W\to Y)$ is a morphism in $\mathsf{VectBun}_{\mathbb{R}}$,then we have a commutative diagram
\[
\begin{tikzcd}
{X^{V}} \arrow[r, "{\Delta_{V}}"] \arrow[d, "{f^V}"] & {X\otimes X^{V}} \arrow[d, "{f\otimes f^V}"] \\
{Y^W} \arrow[r, "{\Delta_V}"]             & {Y\otimes Y^W}                  
\end{tikzcd}
\]
In other words, the construction of Thom spectra gives a functor
\begin{align*}
\mathsf{VectBun}_{\mathbb{R}} &\to \mathsf{Comod}(\mathsf{Sp})\\
(V\to X)& \mapsto (X,X^V),
\end{align*}
where for any monoidal category $\mathsf{C}$ we denote by $\mathsf{Comod}(\mathsf{C})$ the category whose objects are pairs $(A,M)$, where $A$ is a coalgebra object in $\mathsf{C}$ and $M$ is an $A$-comodule, and morphisms from $(A,M)$ to $(B,N)$ are commutative diagrams
\[
\begin{tikzcd}
{M} \arrow[r, "{\Delta_M}"] \arrow[d, "\hat{f}"] & {A\otimes M} \arrow[d, "{f\otimes \hat{f}}"] \\
{N} \arrow[r, "{\Delta_N}"]             & {B\otimes N}                  
\end{tikzcd}
\]
where $f\colon A\to B$ is a morphism of coalgebras.

This gives a natural $[X,E]$-module structure on the $E$-cohomology of $X^V$.

\begin{remark}\label{rem:on-thom-spectra}
If $f\colon X\to Y$ is a morphism of (nice) topological spaces and $V$ is a vector bundle on $Y$ we have a naturally induced morphism 
\[
(f,f^V)\colon (X,X^{f^*V}) \to (Y,Y^V)
\]
in $\mathsf{Comod}(\mathsf{Sp})$.
In particular, given two vector bundles $V$ and $W$ over a space $X$, we have a natural morphism
\[
(\Delta,\Delta^{V\boxplus W})\colon (X,X^{V+W}) \to (X\otimes X,X^V\otimes X^W)
\]
as a consequence of the pullback isomorphism $V+W=\Delta^*(V\boxplus W)$, where $\Delta\colon X\to X\times X$ is the diagonal embedding.
\end{remark}

\begin{remark}\label{rem:virtual-bundles} The category  $\mathsf{VectBun}_{\mathbb{R}}$ has a positive shift given by tensoring with the trivial rank 1 line bundle $\mathbb{R}\to \ast$. As we have seen in Remark \ref{rem:positive-shift}, the Thom spectrum functor changes the positive shift on $\mathsf{VectBun}_{\mathbb{R}}$ into the shift functor on $\mathsf{Sp}$. As a consequence, we have a well defined notion of Thom spectrum of a virtual bundle $V=V_0-V_1$ on $X$, given by $X^{V}=X^{V_0+\tilde{V}_1}[-n]$, where $\tilde{V}_1$ is a vector bundle on $X$ such that $V_1+ \tilde{V}_1$ is the trivial rank $n$ vector bundle, for some $n$. As a shifted comodule is again a comodule, we see that $X^V$ is naturally an $X$-comodule for any virtual vector bundle $V$ on $X$.
\end{remark}

\begin{example}\label{ex:collapse-map}
Let $X$ be a compact smooth manifold, and let $i\colon X\hookrightarrow \mathbb{R}^n$ be an embedding. If $\nu$ denotes the normal bundle to $X$ in $\mathbb{R}^n$, then $TX+\nu$ is the trivial rank $n$ vector bundle over $X$, so that $X^{-TX}=X^\nu[-n]$. By considering $\mathbb{R}^n$ inside its one-point compactification $S^n$, one can look at $i$ as an embedding $i\colon X\hookrightarrow S^n$. A tubular neighborhood $U$ of $X$ in $S^n$ is homeomorphic to the open unit disk bundle $B(\nu)$ of the normal bundle $\nu$. Mapping all points of $S^n$ outside $U$ to the point at infinity in the Thom space of $\nu$ while keeping the points of $U$ fixed defines a map of pointed spaces $
S^n\to \mathrm{Th}(\nu)$. Applying $\Sigma^\infty$, this defines a map of spectra
$\SSS[n] \to X^\nu$ 
and so, equivalently, a map of spectra
\[
\gamma^{}_{\mathrm{PT}}\colon \SSS\to X^{-TX}.
\]
The latter is, up to homotopy, independent of all the choices involved in its definition. It is called the \emph{Pontryagin--Thom collapse map}. 
\end{example}

\section{Unde idem dicebant esse dualitatem et speciem dualitatis}\label{sec:integration}
A second main feature of the category of spectra we will use is that it is monoidally closed, i.e., for any object $X$ the functors
\begin{align*}
-\otimes X\colon \mathsf{Sp}&\to   \mathsf{Sp}\\
X\otimes -\colon \mathsf{Sp}&\to   \mathsf{Sp}
\end{align*}
have right adjoints. We we will write
\[
F(X,-)\colon \mathsf{Sp}\to   \mathsf{Sp}
\]
for the right adjoint of the right multiplication functor $-\otimes X$, i.e., we have natural isomorphisms
\[
\mathsf{Sp}(Y\otimes X,Z)\cong \mathsf{Sp}(Y,F(X,Z))
\]
for any spectra $X,Y,Z$.

\begin{definition}
The \emph{Alexander--Spanier dual} $DX$ of a spectrum $X$ is 
\[
DX:=F(X,\mathbb{S}),
\]
i.e., it is the spectrum characterized by 
\[
\mathsf{Sp}(Y,DX)\cong
\mathsf{Sp}(Y\otimes X,\mathbb{S}),
\]
naturally, for any spectrum $Y$.
\end{definition}
It is immediate from the definition that $D$ is a contravariant functor $
D\colon \mathsf{Sp}\to \mathsf{Sp}^{\mathrm{op}}$.
Moreover, as $\SSS$ is the unit object, we have a natural isomorphism $
D\mathbb{S}\cong \mathbb{S}$.
In particular, for any space $X$ we have a distinguished morphism
\[
\varphi\colon \mathbb{S}\to DX,
\]
which is the image under the duality functor $D$ of the distinguished morphism $X\to \mathbb{S}$.

We have already noticed (nice) topological spaces are special among spectra. A somehow surprising fact is that, even from the spectral point of view, compact smooth manifolds are special among (nice) topological spaces. We have the following
\begin{theorem}[Atiyah]\cite{doi:10.1112/plms/s3-11.1.291}\label{theorem:DX-thom} Let $X$ be a compact smooth manifold. Then there is an isomorphism of spectra under $\mathbb{S}$
\[
\begin{tikzcd}
 & \mathbb{S} \arrow[rd,"\varphi"]\arrow[ld, "\gamma^{}_{\mathrm{PT}}"']\\
 X^{-TX} \arrow[rr, "{\sim}"] && DX
\end{tikzcd}
\]
where $TX$ is the tangent bundle of $X$, $X^{-TX}$ is the \emph{Thom spectrum} of the virtual bundle $-TX$, and $\gamma^{}_{\mathrm{PT}}\colon \mathbb{S}\to X^{-TX}$ is the Pontryagin--Thom collapse map from Example \ref{ex:collapse-map}.
\end{theorem}

From Theorem \ref{theorem:DX-thom}, we get
\begin{corollary}
Let $X$ be a compact smooth manifold. Then $DX$ carries a natural $X$-comodule structure. In particular, if $E$ is a ring spectrum, then $[DX,E]$ carries a natural $[X,E]$-module structure.
\end{corollary}
We can now give the main definition of this Section:
\begin{definition}\label{defEori}
Let $E$ be a commutative ring spectrum. A compact $n$-dimensional smooth manifold $X$ is \emph{$E$-orientable} if $[DX[n],E]$ is an $[X,E]$-module of rank 1. An \emph{$E$-orientation} of $X$ is an isomorphism of $[X,E]$-modules $[X,E]\to [DX[n],E]$. An \emph{$E$-oriented} compact $n$-dimensional smooth manifold $X$ is a pair $(X,\sigma)$ where $X$ is a compact $n$-dimensional smooth manifold and $\sigma:[X,E]\to[DX[n],E]$ is an isomorphism of $[X,E]$-modules.
\end{definition}
Notice that, if $(X,\sigma)$ is an $E$-oriented compact $n$-dimensional smooth manifold, then the datum of $\sigma$ is equivalent to the datum of a generator $\tau = \tau_{(X,\sigma)}$ of $[DX[n],E]$ as an $[X,E]$-module. The element $\tau$ will be called the \emph{Thom} class of the $E$-orientation of $X$. 
\begin{definition}\label{integration}
For $(X,\sigma)$ an $E$-oriented compact $n$-dimensional smooth manifold we have a naturally defined \emph{integration map}

\[
\int_{(X,\sigma)}^E\colon [X,E] \to [\mathbb{S}[n], E]
\]

\noindent given by the composition

\[
\begin{tikzcd}
\left[X,E\right] \arrow[r,"\sigma"] & \left[DX[n],E\right]   \arrow[r,"\varphi^*"] & \left[\SSS[n],E\right]
\end{tikzcd}
\]

\end{definition}

\begin{remark}
When $E=\bigoplus_{i\in\Z}H\mathbb{Z}[i]$, where $H\mathbb{Z}$ is the Eilenberg-Mac Lane spectrum  representing integral cohomology, the datum of an $E$-orientation of a compact smooth manifold $X$ is equivalent to the datum of an orientation of $X$ in the sense of differential geometry, and the integration map defined above is naturally identified with integration in singular cohomology.
\end{remark}

\section{Cubicula et bibliothec\ae\ ad orientem spectare debent}
So far we have been considering $E$-orientations of compact smooth manifolds, defined in terms of Thom spectra of opposite tangent bundles. It will be convenient to generalize this construction to arbitrary real virtual vector bundles (recall from Remark \ref{rem:virtual-bundles} that we have a well defined notion of Thom spectrum $X^V$ for any real virtual vector bundle $V$ over a (nice) topological space $X$).
\begin{definition}
Let $V$ be a rank $r$ virtual real vector bundle over $X$. $V$ is is said to be $E$-{{\it orientable}} if $[X^V[-r], E]$ is 
an $[X,E]$-module of rank $1$. An $E$-{{\it orientation}} of $V$ is the choice of an isomorphism $\sigma_V \colon [X,E] \xrightarrow{\sim} [X^V[-r], E]$. 
\end{definition}
\begin{remark}\label{rmk.trivialori}
Notice that with this definition each trivial real vector bundle $\mathbf{n}=X\times\R^n\to X$ has a canonical $E$-orientation. Namely, $X^{\mathbf{n}} = X[n]$, and so the shift isomorphism provides a distinguished isomorphism $[X,E] \to [X^{\mathbf{n}}[-n],E]$.
\end{remark}

\begin{definition}
The datum of an $E$-orientation $\sigma_V$ of a rank $r$ virtual vector bundle  $V$ is equivalent to the datum of a generator $\tau_V$ of $[X^V[-r], E]$ as an $[X,E]$-module; namely, the image of the unit $1\in[X,E]$ via the isomorphism $\sigma_V$. The element 
\[
\tau_V\colon X^V[-r] \to E
\]
will be called the \emph{Thom element} of the $E$-oriented bundle $V$.
\end{definition}
\begin{remark}\label{multiplier}
Since orientations are isomorphisms of $[X,E]$-modules from $[X,E]$, they are a torsor for the group of $[X,E]$-module automorphisms of $[X,E]$, i.e. for the group $GL_1[X,E]$ of units of $[X,E]$. In other words, if $\sigma_V, \tilde{\sigma}_V$ are two $E$-orientations on an $E$-orientable rank $r$ (virtual) vector bundle $V$ over $X$, there exists a unique element $m_V$ in $GL_1[X,E]\subseteq [X,E]$ making the following diagram
\[
\begin{tikzcd}[column sep=tiny]
{[X,E]} \arrow[rr, "{ m_V\cdot-}"] \arrow[rd, "\sigma_V"'] &              & {[X,E]} \arrow[ld, "\tilde{\sigma}_V"] \\
                                                                                         & {[X^V[-r],E]}, &                                       
\end{tikzcd}
\]
commute, where $\cdot$ is the product in $[X,E]$. Equivalently, we have $\tau_V=m_V\cdot \tilde{\tau}_V$.
\end{remark}

\begin{remark}\label{rem:GL1-representable}
The functor 
$X\mapsto GL_1[X,E]$ is representable: there exists a spectrum $GL_1E$ with a natural isomorphism
\[
GL_1[X,E] \cong [X,GL_1E].
\]
Therefore a multiplier $m_V$ can be equivalently seen as a morphism $m_V\colon X\to GL_1E$.
\end{remark}

\begin{remark}\label{rem:2-out-of-3}
A noteworthy property of $E$-orientations is that they satisfy the 2-out-of-3 property: if a short exact sequence $0\to V_1\to V_2\to V_3\to0$ of vector bundles on a topological space is given, then $E$-orientations on two of the vector spaces in the sequence canonically determine an $E$-orientation of the third one, see \cite{May-E-infty}. 
\end{remark}
The main idea now is to study not only an $E$-orientation of a single vector bundle, but to look at systems of coherent $E$-orientations of a family of vector bundles.

\begin{definition}
A family of vector bundles $\mathcal{F}$ is said to be \emph{closed} if it closed under the operations of pullback and box sum. This means that:
\begin{enumerate}
    \item if the vector bundles $V\to X$ and $W\to Y$ are in the family $\mathcal{F}$, then the box sum $V\boxplus W\to X\otimes Y$ is in $\mathcal{F}$,\\
    \item for any map $f:X \to Y$ and every vector bundle $V\to Y$ in the family $\mathcal{F}$, then the pullback bundle $f^*V\to X$ is in $\mathcal{F}$.
    \end{enumerate}
\end{definition}
\begin{remark}
Notice that, if $V,W$ are vector bundles over $X$, then via the isomorphism $\Delta^*(V\boxplus W) \cong V+W$ from Remark \ref{rem:on-thom-spectra} any closed family is automatically closed under the operation of direct sum of vector bundles. 
\end{remark}
\begin{example}
Examples of closed families are the family of all (finite rank) real vector bundles and the family of all (finite rank) trivial real vector bundles. Other, more interesting examples are given by the family of all
real oriented vector bundles (i.e. real vector bundles $V$ with a reduction of the structure group to $SO(\rk V)$, where $\rk{V}$ denotes the rank of $V$ as a real vector bundle), spin bundles (i.e, with reduction of the structure group to $Spin(\rk V)$), and 
complex vector bundles (i.e., even rank real vector bundles $V$ with a reduction of the structure group to $U(\frac{1}{2}\rk V)$).
\end{example}

\begin{definition}\label{def:oriented-family}
A \emph{coherent system of} $E$-\emph{orientations} on a closed family of vector bundles $\mathcal{F}$ (or an $E$-orientation of $\mathcal{F}$) is the datum of an $E$-orientation $\sigma_V: [X,E] \xrightarrow{\sim} [X^V[-\rk{V}], E]$, for each $V \in \mathcal{F}$, satisfying the following coherence conditions: 

\begin{enumerate}
\item Given $V,W\in\mathcal{F}$ and $f:X\to Y$ the following diagrams commute.
 \[
\begin{tikzcd}
{[X,E]\otimes [Y,E]} \arrow[r, "\sigma_V\otimes \sigma_W"] \arrow[d] & {[X^V[-\rk{V}],E] \otimes [Y^W[-\rk{W}],E]} \arrow[d] \\
{[X\otimes Y,E]} \arrow[r, "\sigma_{V\boxplus W}"]                   & {[X^{V}\otimes Y^W[-\rk{V}-\rk{W}],E ]}              
\end{tikzcd}
\]
and
\[
\begin{tikzcd}
{[Y,E]} \arrow[d, "f^*"'] \arrow[r, "\sigma_V"] & {[Y^{V}[-\rk{V}],E]} \arrow[d, "(f^V)^*"] \\
{[X,E]} \arrow[r, "\sigma_{f^*V}"']             & {[X^{f^*V}[-\rk{f^*V}],E]}                  
\end{tikzcd}
\]
\vskip 0.5 cm 
\item if a trivial vector bundle $\mathbf{n}$ is in $\mathcal{F}$, then $\sigma_{\mathbf{n}}$ is the canonical orientation of $\mathbf{n}$.
\end{enumerate}
\end{definition}

The next definition is motivated by the following two facts we noticed in Remarks \ref{rmk.trivialori} and \ref{rem:2-out-of-3}: every trivial bundle is canonically oriented, and $E$-orientations satisfy a $2$-out-of-$3$ property.
\begin{definition}
A closed family of vector bundles $\mathcal{F}$ is said to be \emph{stable} if it contains the family of trivial bundles and satisfies the $2$-out-of-$3$ property, i.e., if for a short exact sequence of vector bundles
\[
0 \to V_1 \to  V_2 \to V_3 \to 0
\]
two of the $V_i$'s are in $\mathcal{F}$ then also the third is in $\mathcal{F}$.
\end{definition}
\begin{example}
Let us consider the family of complex vector bundles. This is a closed family but not a stable family: it contains only ``half'' of the trivial bundles (those with even rank as real vector bundles). To obtain a stable family out of the closed family of complex bundles we need to stabilize it, i.e. consider vector bundles $V$ such that $V+\mathbf{k}$ is complex for some $k\geq 0$. Such bundles are called stably complex. The family of stably complex vector bundles is a stable family and it is the smallest stable family that contains the family of complex bundles. More generally, for any closed family of vector bundles we can consider its stabilization.
\end{example}

\begin{remark}
An $E$-orientation of a stable family $\mathcal{F}$ is defined as an $E$-orientation of $\mathcal{F}$ as a closed family. It is immediate to see that the datum of an $E$-orientation of a closed family is equivalent to the datum of an $E$-orientation on its stabilization.
\end{remark}

\begin{remark}\label{rem:oriented-family}
Equivalently, conditions (1--2) above can be expressed in terms of Thom elements as
\begin{enumerate}
    \item $\tau_{V\boxplus W}=\tau_V \otimes_E \tau_W$;
    \item $\tau_{f^*V}=(f^V)^*\tau_V$;
    \item $\tau_{\mathbf{n}}=1\in [X,E]=[X^{\mathbf{n}}[-n],E]$.
\end{enumerate}
\end{remark}
\begin{remark}
Let $\{\sigma_V\}_{V \in \mathcal{F}}$ be a an $E$-orientation of a stable family $\mathcal{F}$. Then it makes sense to talk about the $E$-orientation $\sigma_{V}$ of a virtual bundle $V=V_0-V_1$ where both $V_0,V_1 \in \mathcal{F}$. Namely, by the $2$ out of $3$ property, any complement $\tilde{V}_1$ of $V_1$ is in $\mathcal{F}$. Writing $V = V_0+\tilde{V}_1-\mathbf{n}$, we have natural isomorphisms of shifted Thom spectra
\[
X^{V}[-\rk{V}] \cong X^{V_0+\tilde{V}_1-\mathbf{n}}[-\rk{V}] \cong
\]
\[
\cong  X^{V_0+\tilde{V}_1}[-\rk{V}-n] \cong X^{V_0+\tilde{V}_1}[-\rk{(V_0+\tilde{V}_1)}].
\]
Since $V_0+\tilde{V}_1 \in \mathcal{F}$, we have an $E$-orientation $\sigma_{V_0+\tilde{V}_1}$, which we can take as the $E$-orientation $\sigma_{V}$ of $V$. One checks that $\sigma_{V}$ is well defined, i.e., it is independent of the choice of the complement $\tilde{V}_1$. This way the notion of $E$-orientation of a stable family of vector bundles extends to stable families of virtual vector bundles. 
\end{remark}

\begin{remark}\label{rem:boxsum}
Let $V,W$ be two (virtual) vector bundles over the manifold $X$ in the $E$-oriented family $\mathcal{F}$. Then the pullback isomorphism $V+ W = \Delta^*(V\boxplus W)$,
and conditions (1--2) in Remark \ref{rem:oriented-family} give 
\[
\tau_{V+W}=(\Delta^{V\boxplus W})^*(\tau_V \otimes_E\tau_W)=:\tau_V \cdot \tau_W.
\]
The normalisation condition (3) then gives $\tau_V\cdot\tau_{-V}=1$.
\end{remark}

\begin{remark}
It follows from conditions (1--2) in Definition \ref{def:oriented-family} that for any (virtual) vector bundle $V$ in the $E$-oriented stable family $\mathcal{F}$ on $X$, the multiplication by the Thom class $\tau_V\in [X^V[-\rk V],E]$ gives natural morphisms of $[X,E]$-modules
\[
[X^W,E]\xrightarrow{\tau_V\cdot } [X^{V+W}[-\rk V],E].
\]
As $\tau_V\cdot\tau_{-V}=1$, these are actually isomorphisms. In particular, by taking $W=\mathbf{n}$, we see that  multiplication by the Thom class $\tau_V$ is an isomorphism of 
$[X,E]$-modules
\[
[X[n],E]\xrightarrow{\tau_V\cdot } [X^{V}[n-\rk V],E],
\]
for
any $n\in \mathbb{Z}$.
\end{remark}

\medskip

So far we have defined compatible systems of $E$-orientations of a stable family of vector bundles $\mathcal{F}$. Again, rather than study a single system, it is interesting to study how two compatible systems interact. As in Remark \ref{multiplier}, two systems of $E$-orientations $\{\sigma_V,s_V\}_{V \in \mathcal{F}}$ define a set of Thom classes $\{\tau_V,t_V\}_{V \in \mathcal{F}}$ as well as a set of multipliers $\{m_V\}_{V \in \mathcal{F}}$, uniquely defined by the property 
$\tau_V = m_Vt_V$. Now we want to state some properties satisfied by the set multipliers $\{m_V\}_{V \in \mathcal F}$.
\begin{proposition}\label{prop3.8}
Let $\mathcal{F}$ be a stable family of vector bundles, and $\sigma$ and $s$ two $E$-orientations of $\mathcal{F}$, with associated Thom classes $\tau$ and $t$, and system of multipliers $m$. Then for any two (virtual) vector bundles in the family $\mathcal{F}$ over a space $X$, we have $m_{V+W} = m_Vm_W$, and for any trivial vector bundle $\mathbf{n}$ on $X$ we have $m_{\mathbf{n}}=1$. In particular, $m_{V+\mathbf{n}}=m_V$ for any $V$ and $\mathbf{n}$.
\end{proposition}
\begin{proof}
As the $X$-comodule structure on $X^{V+W}$ and the $X\otimes X$-comodule structure on $X^V\otimes X^W$ are compatible (Remark \ref{rem:on-thom-spectra}), for any $\lambda, \mu \in [X,E]$, we have  $(\lambda\mu) (t_V\cdot t_W) = \lambda t_V \cdot \mu t_W$. Therefore, by Remark \ref{rem:boxsum} we have that
\begin{align*}
m_{V+W} t_{V+W}&=
\tau_{V+W}\\ &= \tau_V \cdot \tau_W\\
&= m_Wt_V \cdot m_Wt_W\\
&= m_Vm_W t_V \cdot t_W\\
&= m_Vm_Wt_{V+W}. 
\end{align*}
The conclusion follows from uniqueness of multipliers. The statement for trivial bundles is immediate from the normalisation condition $\tau_{\mathbf{n}}=t_{\mathbf{n}}=1$.
\end{proof}

\begin{corollary}
For the multipliers associated to a (virtual) vector bundle $V \in \mathcal{F}$ and to its opposite $-V$, we have $m_Vm_{-V} = 1$.
\end{corollary}

\begin{remark}
By Remark \ref{rem:GL1-representable}, the collection of multipliers can be seen as a collection of morphisms $m_V\colon X\to GL_1E$, for any (virtual) vector bundle $V$ over $X$,  and the compatibility of multipliers with pullbacks amounts to the commutativity of the diagrams
\[
\begin{tikzcd}
{Y} \arrow[rr, "{f}"] \arrow[rd, "m_{f^*V}"'] &              & {X} \arrow[ld, "m_V"] \\
                                                                                         & {GL_1E}, &                                       
\end{tikzcd}
\]
for any map $f:Y\to X$.
\end{remark}

\section{Inque tuo sedisti, Sisyphe, saxo}
We can now introduce the notion of an $E$-oriented map between compact smooth manifolds and define a pushforward in $E$-cohomology along an $E$-oriented map. This will generalize the construction of the integration map in Section \ref{sec:integration}, which will be recovered as the pushforward along the terminal morphism $X\to \ast$ for an $E$-oriented manifold $X$.

For a smooth map $f\colon X\to Y$ we write $T_f$ for the virtual bundle over $X$ defined by
\[
T_f:=TX-f^*TY.
\]
\begin{definition}
An $E$-orientation of a smooth map $f\colon X\to Y$ is an $E$-orientation of the virtual bundle $-T_f$.
\end{definition}
\begin{remark}
Notice that an $E$-orientation of a manifold is a special case of $E$-orientation for a morphism: $X$ is $E$-oriented precisely when $t:X\to *$ is.
\end{remark}

\begin{remark}
The 2-out-of-3 property of $E$-orientation of (virtual) vector bundles implies that also $E$-orientations of maps have a 2-out-of-3 property: if 
\[
X\xrightarrow{f} Y\xrightarrow{g} Z
\]
are smooth maps and two out of $\{f,g,g\circ f\}$ are $E$-oriented then so is the third. In particular, if $f\colon X\to Y$ is a smooth map between $E$-oriented manifolds then $f$ is canonically $E$-oriented. See \cite{May-E-infty}.
\end{remark}

We now describe how to define a pushforward map 
\[
f_*\colon [X,E]\to [Y[\dim X-\dim Y],E] 
\]
for an $E$-oriented map $f\colon X \to Y$ which is a smooth fibration with typical fibre a smooth compact manifold $F$ of dimension $d=\dim X-\dim Y$. In this situation we can think of $f$ as a parametrized family over $Y$ of 
(nice) topological spaces, i.e., as an $\infty$-functor $f\colon Y \to \mathsf{Top}$, where on the left we are writing $Y$ for its $\infty$-Poincar\'e  groupoid (i.e., for the $\infty$-groupoid with objects the points of $Y$, $1$-morphisms the paths between points, $2$-morphisms the homotopies between paths, etc), and on the right $\mathsf{Top}$ is the $\infty$-category of (nice) topological spaces, with homotopies between continuous maps, homotopies between homotopies, etc. as higher morphisms. Namely, the definition of (Serre) fibration is precisely a way of encoding this idea. Now, we are always looking at topological spaces as spectra via $\Sigma^\infty_+$, so we look at $f$ as a family of spectra parametrized by $Y$,
\[
Y\xrightarrow{f} \mathsf{Top}\xrightarrow{\Sigma^\infty_+} \mathsf{Sp}.
\]

We refer the interested reader to \cite{abgpara} or \cite{maysig} for a detailed and rigorous treatment of the category $\mathsf{Sp}_Y$ of all such \emph{parametrized spectra} with parameter space $Y$; for the aim of this note the intuitive definition sketched above will be sufficient.

The $\infty$-category $\mathsf{Sp}_Y$ inherits from $\mathsf{Sp}$ the pointwise monoidal structure with unit object the pointwise unit $\SSS_Y$ given by the constant family with fibre $\SSS$ over $Y$, and so pointwise duals. In particular we will have an Alexander--Spanier dual $D_Y(f)$ coming with a distinguished morphism

\[
\varphi_f\colon \SSS_Y \to D_Y(f)
\]
in $\mathsf{Sp}_Y$.
By putting everything together (formally, this means taking the $\infty$-colimit of
the natural transformation $\varphi_f$ of $\infty$-functors $Y\to \mathsf{Sp}$), we get a map
\[
Y \to {{\mathrm{colim}}}_Y D_Y(f).
\]
If we denote by $F_y$ the fibre of $f\colon X\to Y$  over $y\in Y$, then
Atiyah duality pointwise identifies $D_Y(f)_y=D(F_y)$ with $F_y^{-T_f}$ so the one above is a map
\[
Y \to {{\mathrm{colim}}}_Y F_y^{-T_f}.
\]
Maybe not surprisingly, the colimit on the right hand side is the Thom spectrum $X^{-T_f}$ (see \cite{abgpara} for a rigorous proof). Therefore, at least  when $f\colon X \to Y$ is a smooth fibration with compact fibres, we have a natural Pontryagin--Thom morphism
\[
\varphi_{PT}\colon Y \to X^{-T_f}.
\]

In the best of possible worlds, this would be true for any smooth map $f$, not necessarily a fibration with compact fibres. And indeed it is. This can be shown by factoring the map $f$ as the composition of an embedding and a smooth fibration with compact fibres and noticing that for an embedding $\iota:X\hookrightarrow Y$ one has the geometrically defined Pontryagin--Thom collapse map $Y \to X^{\nu_\iota} = X^{-T_\iota}$. However, to our knowledge, the general Pontryagin--Thom morphism $Y \to X^{-T_f}$ for a general $f$ has not yet been given a transparent interpretation in terms of the monoidal closedness of the category of spectra. See, however \cite[Remark 4.17]{abgpara} .

The morphism $\varphi_{PT}\colon Y \to X^{-T_f}$ is a morphism of $Y$-comodules, where the $Y$-comodule structure on $X^{-T_f}$ is induced by its $X$-comodule structure via $f$, i.e., the diagram
\[
\begin{tikzcd}
Y \arrow[d, "\Delta_Y"'] \arrow[rr, "\varphi_{PT}"]      &                                                & X^{-T_f} \arrow[d, "(f\otimes \mathrm{id})\circ \Delta_{\-T_f}"]                 \\
Y\otimes Y \arrow[rr, "{{\rm id}} \otimes \varphi_{PT}"'] &  & Y \otimes X^{-T_f}  
\end{tikzcd}
\]
commutes. 

Therefore, if $f$ is $E$-oriented we can define a \emph{pushforward map}
\[
f_*\colon [X,E]\to [Y[\dim X-\dim Y],E]
\]
as the composition
\[
[X,E]\xrightarrow{\sigma_{-T_f}} [X^{-T_f}[\dim X-\dim Y],E]\xrightarrow{\varphi_{PT}^*}[Y[\dim X-\dim Y],E].
\]
When $f$ is the terminal morphism, the pushforward map $f_\ast$ is $E$-integration over $X$.

\begin{remark}
Since $\varphi_{PT}$ is a morphism of $Y$-comodules, 
\[
\varphi^*_{PT}\colon [X^{-T_f}[\dim X-\dim Y],E]\to [Y[\dim X-\dim Y],E]
\]
is a morphism of $[Y,E]$-modules. Since $[X,E]\xrightarrow{\sigma_{-T_f}} [X^{-T_f}[\dim X-\dim Y],E]$ is an isomorphism of $[X,E]$-modules by definition of $E$-orientation, it will also be an isomorphism of $[Y,E]$-modules where we look at every $[X,E]$ module as an $[Y,E]$-module via the ring homomorphism $f^*\colon [Y,E]\to [X,E]$. Summing up, $f_*$ is a morphism of $[Y,E]$-modules, where we look at $[X,E]$ as an $[Y,E]$-module via $f^*$. This is the \emph{projection formula}
\[
f_*((f^*a)\cdot x)=a\cdot f_*(x)
\]
for any $a\in [Y,E]$ and $x\in [X,E]$.
\end{remark}

\begin{remark}\label{rem:pullback-duality}
As we noticed, if  $X$ and $Y$ are $E$-oriented, then $f\colon X\to Y$ gets a canonical $E$-orientation by the 2-out-of-3 property. If moreover $X$ and $Y$ are compact, multiplications by the Thom classes of $TY$ and of $f^*TY$ intertwine the morphisms induced in $E$-cohomology by $\varphi_{PT}$ and the dual of $f$, i.e., we have a commutative diagram
\[
\begin{tikzcd}
{[DX[\dim X],E]} \arrow[d, "\wr"'] \arrow[rr, "(Df)^*"]      &                                                & {[DY[\dim X],E]} \arrow[d, "\wr"]                 \\
{[X^{-TX}[\dim X],E]}\arrow[d,  "\wr","f^*\tau_{TY}"'] && {[Y^{-TY}[\dim X],E]}\arrow[d, "\tau_{TY}", "\wr"']  \\
 {[X^{-T_f}[\dim X-\dim Y],E]} \arrow[rr, "\varphi_{PT}^*"'] &  & {[Y[\dim X-\dim Y],E]}  
\end{tikzcd}
\]
of $[Y,E]$-modules. So we see that under these assumptions the pushforward $f_*$ can be written as the composition
\[
[X,E]\xrightarrow{\sim} [DX[\dim X],E] \xrightarrow{(Df)^*} [DY[\dim X],E] \xrightarrow{\sim} [Y[\dim X-\dim Y],E].
\]
\end{remark}

\begin{remark}
Since they are pullbacks along dualized morphisms, pushforwards are covariantly functorial:  if we are given a composition 
$X\xrightarrow{f} Y\xrightarrow{g} Z$ of $E$-oriented maps, then we have $(g\circ f)_*=g_*\circ f_*$. This is immediately seen via Remark \ref{rem:pullback-duality} in case $X,Y$ and $Z$ are compact $E$-oriented manifolds, 
and by a similar argument in the general case of $E$-oriented maps. See  \cite{dyer} for details.

\end{remark}

\section{E pluribus unum}
Let us consider complex vector bundles. As every rank $k$ complex vector bundle over a manifold $X$ is a pullback of the tautological vector bundle $V_k\to BU(k)$, the naturality of orientations with respect to pullbacks tells us that, in order to $E$-orient all complex vector bundles, we only need to orient all the tautological bundles $V_k\to BU(k)$. Moreover, from
\[
X^{V+\mathbf{1}_{\mathbb{C}}}=X^V[2],
\]
where $\mathbf{1}_{\mathbb{C}}:=\mathbb{C}\times X \to X$ is the trivial rank 1 complex vector bundle over $X$,
and from $j^*V_{k+1}=V_k+\mathbf{1}_{\mathbb{C}}$, where $j\colon BU(k) \to BU(k+1)$ is the canonical embedding, we see that a coherent system of $E$-orientations on all of the $V_k$'s is equivalent to the datum of a commutative diagram
\[
\begin{tikzcd}
\cdots \arrow[r] & {MU(k-1)}\arrow[dr,"{\tau_{k-1}}"'] \arrow[r] & {MU(k)} \arrow[d,"{\tau_k}"] \arrow[r] & {MU(k+1)} \arrow[dl,"{\tau_{k+1}}"] \arrow[r]& \cdots\\
           &                                         & E                                   
\end{tikzcd}
\]
of maps of spectra,
where we have written $MU(k)$ for $BU(k)^{V_k}[-2k]$. By the universal property of the limit, this is equivalent to a single map of spectra
\[
\rho\colon MU \to E,
\]
where by definition 
\[
MU=\varinjlim MU(k)
\]
This spectrum $MU$ may informally be thought of as the infinite desuspension of the Thom spectrum of the infinite dimensional tautological bundle over $BU$.

So far we have not used compatibility of orientations with formation of direct sums, the latter operation giving  the $\infty$-abelian group structure on $BU$. Requiring this is equivalent to requiring that $\rho\colon MU \to E$ is a morphism of homotopy commutative ring spectra. 
The morphism $\rho$ described above has by construction the rather special property that all of its components $\rho_k = \tau_{k}\colon MU(k)\to E$ are generators of $[MU(k),E]$, which are free rank 1 $[BU(k),E]$-modules. Quite remarkably, this property is actually not special at all: \emph{every} morphism of homotopy commutative ring spectra $\rho\colon MU\to E$ has this property. Namely, one has $MU(0)=\SSS$ and the morphism $MU(0)\to MU$ is the unit of $MU$. As a ring morphism preserves the unit, we have a commutative diagram
\[
\begin{tikzcd}
{\SSS=MU(0)}\arrow[dr] \arrow[r] & {MU(1)} \arrow[d,"{\rho_1}"] \arrow[r] & \cdots  \arrow[r]& {MU}\arrow[dll,"{\rho}"]\\
           &                                          E                                   
\end{tikzcd}
\]
Now we have the following
\begin{lemma}\label{lemma:negativity}
The zero section $\iota\colon BU(1) \to BU(1)^{V_1}$ is a homotopy equivalence.
\end{lemma}
\begin{proof}
Let $L_{1;n}=\mathcal{O}_{\mathbb{P}^n\C}(1)$ be the universal line bundle over $\mathbb{P}^n\C$. 
There is a natural isomorphism $({\mathbb{P}^n\C})^{L_{1;n}} \cong \mathbb{P}^{n+1}\C$ such that the zero section $\mathbb{P}^n\C \to ({\mathbb{P}^n\C})^{V_{1;n}} \cong \mathbb{P}^{n+1}\C$ is identified with the standard inclusion $\mathbb{P}^n\C \hookrightarrow \mathbb{P}^{n+1}\C$ as the hyperplane at infinity. Namely, let $p$ be a point in $\mathbb{P}^{n+1}\C\setminus \mathbb{P}^n\C$. Then the collection of projective lines through $p$, with the point $p$ removed, is a holomorphic line bundle over $\mathbb{P}^n\C$ which is immediately seen to be isomorphic to $\mathcal{O}_{\mathbb{P}^n\C}(1)$. This realises the above mentioned identification. Taking the limit over $n$ we get an isomorphism $BU(1)\cong BU(1)^{L}$, where $L=\mathcal{O}(1)$ is the universal line bundle over $BU(1)\cong \PP$. The universal line bundle $L$ is obtained from the tautological line bundle $V_1$ by pullback along the equivalence $\mathrm{inv}\colon BU(1)\to BU(1)$ mapping each line bundle to its inverse induced by the group automorphism of $U(1)$ mapping $z$ to $z^{-1}$. By Remark \ref{rem:on-thom-spectra} we therefore have a commutative diagram
\[
\begin{tikzcd}
{BU(1)}\arrow[d,"\mathrm{inv}"] \arrow[r,"\iota"] & BU(1)^{L} \arrow[d,"\mathrm{inv}^{V_1}"] \\
   {BU(1)}\arrow[r,"\iota"]        & {BU(1)^{V_1}}
\end{tikzcd}
\]
where three of the arrows are homotopy equivalences, and so also the fourth is.
\end{proof}
By the argument in Lemma \ref{lemma:negativity} we see that $\rho_1$ provides an extension $\varepsilon$ of the unit of $E$ to $\PP[-2]$:
\[
\begin{tikzcd}
& {\PP[-2]}\arrow[d,"\mathrm{inv}","\wr"'] \arrow[ddd, dashed, bend left=50, "\varepsilon"]\\
{\mathbb{P}^1\C[-2]}\arrow[ru, "i_1"] \arrow[d,equal] & {\PP[-2]}\arrow[d,"\iota","\wr"'] \\
{\SSS}  \arrow[r] \arrow[rd] & {MU(1)} \arrow[d, "\rho_1"] \\
                                             & E                                   
\end{tikzcd},
\]
where $i_1\colon \mathbb{P}^1\C\to \PP$ is the standard inclusion.
As a consequence of the fact that $\PP$ has a CW-complex structure with only even dimensional cells, one can show that the datum of such an extension $\varepsilon$ gives an isomorphism of $[\PP,E]$-modules $[\PP,E]\cong [\PP[-2],E] \cong [MU(1),E]$ and that $\varepsilon$ as an element in $[\PP[-2],E]$ is a generator of this $[\PP,E]$-module. This is a generalized version of the usual isomorphism $H^2(\PP;\mathbb{Z})\cong H^0(\PP,\mathbb{Z})$ in singular cohomology, see \cite[Lecture  4]{luriechromatic} for details. Therefore, $\rho_1$ is an $E$-orientation of the tautological line bundle $V_1\to BU(1)$. To conclude, we need to show that also the $\rho_k$'s with $k\geq 2$ are $E$-orientations of the tautological vector bundles $V_k\to BU(k)$. This is a corollary of the splitting principle for complex bundles in $E$-cohomology, for $E$ an $E_\infty$-ring spectrum. Namely,  the pullback of the tautological bundle $V_k\to BU(k)$ along the morphism of topological spaces $j_k\colon BU(1)^k\to BU(k)$ splits as $j_k^*V_k=\boxplus_{i=1}^k V_1$ and 
this induces identifications
\[
j_k^*\colon [BU(k),E]\xrightarrow{\sim} [BU(1)^{\otimes k},E]^{\mathrm{Sym}_k}
\]
\[
(j_k^{V_k})^*\colon [MU(k),E]\xrightarrow{\sim} [MU(1)^{\otimes k},E]^{\mathrm{Sym}_k}
\]
where $(-)^{\mathrm{Sym}_k}$ denotes the $\mathrm{Sym}_k$-invariants.
As the $j_k$'s and the actions of the symmetric groups are compatible with the canonical embeddings $BU(k)\to BU(k+1)$, this identifies $\rho_k$ with the symmetric element $(\rho_1)^{\otimes k}$. The latter is manifestly the datum of an isomorphism $[BU(1)^{\otimes k},E]^{\mathrm{Sym}_k}\cong [MU(1)^{\otimes k},E]^{\mathrm{Sym}_k}$ of $[BU(1)^{\otimes k},E]^{\mathrm{Sym}_k}$-modules, so the $\rho_k$'s are $E$-orientations. Again, see \cite[Lectures 4 \& 6]{luriechromatic} for details. 

Summing up, by the above reasoning we have proven the following
\begin{proposition}\label{prop:complex-orientation-spectrum}
A system of compatible $E$-orientations on complex (and so also on stably complex) vector bundles is equivalent to the datum of a single morphism of homotopy commutative ring spectra $\rho\colon MU\to E$.
\end{proposition}
In view of Proposition \ref{prop:complex-orientation-spectrum} it is natural to give the following definition.
\begin{definition}
Let $E$ be an $E_\infty$-ring spectrum. 
A morphism of homotopy commutative ring spectra $\rho\colon MU\to E$ is called a \emph{complex orientation} of the $E_\infty$-ring spectrum $E$.
\end{definition}
\begin{remark}
In particular, the identity morphisms $\mathrm{id}_{MU}\colon MU\to MU$ defines the canonical complex orientation of $MU$.
\end{remark}
\begin{remark}
As both $MU$ and $E$ are $E_\infty$-ring spectra, one may be tempted to think the morphism $MU\to E$ defined by a system of compatible $E$-orientations on complex  vector bundles is actually a morphism of $E_\infty$-ring spectra, but it is actually not necessarily so. Namely, differently from the case of commutative rings inside all rings, the $\infty$-category of $E_\infty$-ring spectra is not a full $\infty$-subcategory of the $\infty$-category of ring spectra. This is so because the enhancement of a ring spectra morphism between two $E_\infty$-rings to an $E_\infty$-ring morphism is structure and not property: it is the additional datum of all the coherent homotopies involved in the definition of a morphism of $E_\infty$-ring spectra. By the same argument for commutative rings inside all rings, a morphism of ring spectra between $E_\infty$-rings is automatically a morphism of homotopy commutative ring spectra, but an enhancement to a full morphism of $E_\infty$-rings may not exist. See \cite{hopkins-lawson} for a detailed discussion. Clearly, if $MU\to E$ is a morphism of $E_\infty$-ring spectra then it is in particular a morphism of homotopy commutative ring spectra, and so a complex orientation of $E$.
\end{remark}

\begin{remark}\label{remark:need-only-rho1}
The above discussion shows that any lift of $1\in [\SSS,E]$ to an element $\varepsilon\in [\PP[-2],E]$ under the pullback morphism $[\PP[-2],E] \to [\mathbb{P}^1\C[-2],E]=[\SSS,E]$ defines a morphism $\rho_1\colon MU(1)\to E$ that
uniquely extends to a morphism of $E_\infty$-ring spectra $\rho\colon MU\to E$. Therefore, (homotopy classes of) complex orientations of $E$ bijectively correspond to these lifts. Moreover, one sees by construction that $\varepsilon$ is identified with the pullback of the Thom class $\tau_L\in [BU(1)^L[-2],E]$ along the zero section $\iota\colon BU(1)\to BU(1)^L$
\end{remark}

\begin{remark}\label{rem:geometric-MU}
Although we are not going to make use of this fact, it is worth mentioning that the spectrum $MU$ has an interesting geometric characterization: it is the \emph{complex cobordism} spectrum. For $X$ a finite dimensional smooth manifold the hom-set $[X[n],MU]$ is the set of complex cobordism classes of proper complex oriented maps $f:Z\to X$ with $\dim f:=\dim Z-\dim X=n$ \cite{quillen1971elementary}. In particular, when $X$ is a point we get the complex cobordism ring
\[
\Omega^U:=\bigoplus_{n\geq 0} [\SSS[n],MU]=\bigoplus_{n\geq 0} \pi_n^{\mathsf{Sp}} MU.
\]
The spectrum $MU$ is $(-1)$-connected, i.e., its homotopy groups vanish in negative degree or, equivalently, $MU\cong MU_{\geq 0}$. This gives
$MU_{\leq 0}\cong \allowbreak (MU_{\geq 0})_{\leq 0} \allowbreak \cong H\pi_0^{\mathsf{Sp}}(MU)$, and 
so the fibre sequence associated to the $0$-truncation of $MU$ is
\[
MU_{>0} \to MU \to MU_{\leq 0} \cong H\pi_0^{\mathsf{Sp}}(MU)=H\Z.
\]
As the 0-truncation morphism for a $E_\infty$-ring spectrum is an $E_\infty$-ring map, we read from the above sequence an $E_\infty$-ring map
\[
MU \to H\Z,
\]
and so a canonical complex orientation for $\Z$-valued singular cohomology. Equivalently, this tells us that the family of all complex vector bundles has a natural theory of Thom classes in $\Z$-valued singular cohomology. We notice that, by the functoriality of
$H-:\mathsf{CommRings}\to\mathsf{E_\infty\text{-}RingSp}$
and since $\Z$ is initial in the category of commutative rings, for every commutative ring $A$ there is a distinguished morphism of $E_\infty$-ring spectra $H\Z\to HA$. As a consequence, singular cohomology with coefficients in $A$ has a canonical complex orientation for every commutative ring $A$.
\end{remark}

\begin{remark}
If $\psi\colon E\to F$ is an homomorphism of homotopy commutative ring spectra, then, a complex orientation $\rho\colon MU \to E$ of $E$ can be pushed forward to a complex orientation $\psi_*\rho$ of $F$. In terms of Thom classes of (virtual) stably complex vector bundles the orientation $\psi_*\rho$ is simply defined by 
\[
\tau^{\psi_*\rho}_V=\psi_*(\tau^\rho_V).
\]
\end{remark}

\begin{remark}
Assume we are given two complex orientations 
\[
\rho_1,\rho_2\colon MU \to E
\]
of a multiplicative cohomology theory $E$. We have seen in Remarks \ref{multiplier} and \ref{rem:GL1-representable} that the collection of multipliers between $\rho_1$ and $\rho_2$ is equivalent to the datum of a compatible family of morphism
$m_V\colon X\to GL_1E$, indexed by vector bundles $V\to X$. Due to compatibility with pullbacks, we can restrict to universal bundles: the datum of the whole collection of multipliers $\{m_V\}$ is equivalent to the datum of the multipliers
\[
m_{k}\colon BU(k) \to GL_1E,
\]
and so to the datum of a commutative diagram
\[
\begin{tikzcd}
\cdots \arrow[r] & {BU(k-1)}\arrow[dr,"{m_{k-1}}"'] \arrow[r] & {BU(k)} \arrow[d,"{m_k}"] \arrow[r] & {BU(k+1)} \arrow[dl,"{m_{k+1}}"] \arrow[r]& \cdots\\
           &                                         & GL_1E                                   
\end{tikzcd}
\]
As $BU=\varinjlim BU(k)$, this is in turn equivalent to the datum of a single morphism
\[
m \colon BU \to GL_1E.
\]
Notice that, since $GL_1[X,E]$ is the multiplicative subgroup of the commutative ring $[X,E]$, the group $GL_1[X,E]$ is an abelian group and so the spectrum $GL_1E$ is an $\infty$-abelian group. As the direct sum of vector bundles is the group operation on $BU$, the equation $m_{V+W}=m_V\cdot m_
W$ implies that $m\colon BU\to GL_1E$ is a morphism of homotopy abelian $\infty$-groups.

Group homomorphisms into an abelian group are themselves an abelian group (more concretely: the bundlewise product of two compatible systems of multipliers is again a compatible system of multipliers), and as we have already noticed multiplying a compatible system of orientations with a compatible system of multipliers one gets a multiplicative system of multipliers. In other words, the space of complex orientations of $E$, i.e., the space of homotopy commutative ring spectra morphisms $MU\to E$ is a torsor over the group $\mathrm{Hom}_{\mathrm{Grp}}(BU,GL_1E)$.
\end{remark}
\begin{remark}
The analysis that we made to establish the equivalence between compatible systems of $E$-orientations on complex vector bundles and homotopy commutative ring spectra maps $MU\to E$ can be done in a completely analogous way to establish an equivalence between compatible systems of $E$-orientations on real (resp.\! oriented real) vector bundles  and homotopy commutative ring spectra maps $MO \to E$ (resp. $MSO\to E$).
By Thom's theorem (See \cite[Theorem 1.5.10]{kochman}), both $MO$ and $MSO$ are cobordism spectra. More precisely, $MO$ is the real cobordism spectrum while $MSO$ is the real oriented cobordism spectrum and the respective cobordism rings are obtained as
\[
\Omega^O\cong\bigoplus_{n\geq 0}[\SSS[n],MO]=\bigoplus_{n\geq 0}\pi_n^{\mathsf{Sp}}MO
\]
and
\[
\Omega^{SO}\cong\bigoplus_{n\geq 0}[\SSS[n],MSO]=\bigoplus_{n\geq 0}\pi_n^{\mathsf{Sp}}MSO
\]
Both $MO$ and $MSO$ are $(-1)$-connected spectra, so the fibre sequences associated to their $0$-truncations are
\[
MO_{>0} \to MO \to MO_{\leq 0} \cong H\pi_0^{\mathsf{Sp}}(MO)=H{\Z/2}
\]
\[
MSO_{>0} \to MSO \to MSO_{\leq 0} \cong H\pi_0^{\mathsf{Sp}}(MSO)=H\Z.
\]
and we get 
two natural $E_\infty$-ring maps
\[
MO \to H{\Z/2}
\]
\[
MSO \to H\Z.
\]
This tells us that the family of all real vector bundles has a canonical orientation/Thom classes in $\Z/2$-valued singular cohomology and the family of all oriented real vector bundles has a canonical orientation/Thom classes in $\Z$-valued singular cohomology.
\end{remark}

\section{Ab uno disce omnis}\label{sec.7}
Assume an $E$-orientation of a stable family $\mathcal{F}$ is given. Thom spectra of actual (i.e., nonvirtual) vector bundles come equipped with natural \emph{zero section} maps $\iota_V\colon X \to X^V$ that are morphisms of $X$-comodules, i.e., the diagrams
\[
\begin{tikzcd}
X \arrow[d, "\Delta"'] \arrow[r, "\iota_{V}"]    & X^{V} \arrow[d, "{\Delta^{V}}"] \\
X \otimes X \arrow[r, "\mathrm{id}\otimes \iota_V"'] & X\otimes X^{V}                 
\end{tikzcd}
\]
commute. So, for any $V$ in $\mathcal{F}$ we have pullbacks
\[
\iota_V^*\colon [X^V[-\rk{V}],E] \to [X[-\rk{V}],E]
\]
which are morphisms of $[X,E]$-modules.
\begin{definition}
The element $e_V := \iota^*\tau_V$ of $[X[-\rk{V}],E]$ is called the \emph{Euler class} of $V$. 
\end{definition}
\begin{remark}
From the commutativity of 
\[
\begin{tikzcd}
X \arrow[d, "\Delta"'] \arrow[r, "\iota_{V+W}"]    & X^{V+W} \arrow[d, "{\Delta^{V\boxplus W}}"] \\
X \otimes X \arrow[r, "\iota_{V}\otimes \iota_W"'] & X^{V}\otimes X^{W}                 
\end{tikzcd}
\]
one gets the multiplicativity of Euler classes, 
\[
e_{V+W} = e_{V}\cdot e_{W},
\]
where $\cdot$ is the product in $[X,E]$. For the trivial bundle $\mathbf{1}$, the zero section $\iota_{\mathbf{1}}\colon X\to X^{\mathbf{1}}=X[1]$ is the inclusion of $X$ into its suspension, so it is homotopically trivial. From this and the multiplicativity of Euler classes it follows that if $V$ has a never zero section, so that $V=V_0+\mathbf{1}$ then, $e_{V}=0$.
\end{remark}

The above general discussion applies in particular to complex orientations of $E$, i.e., to a compatible system of $E$-orientation of complex vector bundles.
By Remark \ref{remark:need-only-rho1}, the datum of such an orientation is equivalent to the the datum of a Thom class $\tau_{L}$ such that, under the isomorphism of $[\mathbb{P}^\infty\C, E]$-modules $[MU(1),E]\cong [\mathbb{P}^\infty\C[-2],E]$ induced by the pullback along the zero section
$\iota\colon BU(1) \to BU(1)^{L}$, it lifts the unit element $1\in [\SSS,E]$. By definition of the Euler class, we therefore have that the Euler class $e_{L} \in [\mathbb{P}^\infty[-2]\C,E]$ is a generator of $[\mathbb{P}^\infty\C[-2],E]$ as a $[\mathbb{P}^\infty\C, E]$-module, with $i_1^*e_{L}=1$, where $i_1\colon \mathbb{P}^1\C\to \mathbb{P}^\infty\C$ is the inclusion.
For a fixed complex orientation on $E$, we will write $e^E$ for the Euler class of the universal line bundle $L=\mathcal{O}(1)$, i.e, we will write $e^E=e_{L}$.
\begin{example}
It follows from the proof of Lemma \ref{lemma:negativity} and by the description of $MU$ given in Remark \ref{rem:geometric-MU} that the Euler class $e^{MU}$ of the canonical complex orientation of $MU$ is the hyperplane inclusion $\mathbb{P}^{\infty-1}\C\hookrightarrow \PP$ seen as a complex oriented proper map of dimension $-2$ to $\PP$, i.e., more precisely, that for any $n\geq 1$ the pullback $i_n^*e^{MU}$ along the standard inclusion $i_n\colon \mathbb{P}^n\C\hookrightarrow \PP$ is the complex oriented proper map of dimension $-2$ to $\mathbb{P}^n\C$ given by the hyperplane inclusion $\mathbb{P}^{n-1}\C\hookrightarrow \mathbb{P}^n\C$. 
\end{example}
 Let now $E$ be an even $2$-periodic $E_\infty$-ring spectrum with Bott element $\beta_E$. One of the most important features of even $2$-periodic spectra, and, in general, of complex orientable spectra, lies in the cohomology of (complex) projective spaces. Since the $E$-cohomology of $\PP$ is defined to be the limit of the $E$-cohomologies of the $\mathbb{P}^n\C$'s over the (pullback of the) inclusions $i_n:\mathbb{P}^n\C \hookrightarrow \mathbb{P}^{n+1}\C$, one gets that the choice of an element $\beta_E \xi$ in $[\PP,E]$ is equivalent to the datum of a compatible sequence $\{\beta_E\xi_n \in [\mathbb{P}^n\C,E]\}_{n \in \N}$, where $\xi_n$ is the pullback of $\xi$ along the inclusion $i:\mathbb{P}^n\C \hookrightarrow \PP$ and hence equivalent to a commutative diagram of rings
\[
\begin{tikzcd}
                 &                                            & {[\SSS,E][u]} \arrow[rd] \arrow[d] \arrow[ld] &                                  &        \\
\cdots \arrow[r] & {[\mathbb{P}^{n+1}\C,E]} \arrow[r, "i_n^*"'] & {[\mathbb{P}^{n}\C,E]} \arrow[r, "i_{n-1}^*"']  & {[\mathbb{P}^{n-1}\C,E]} \arrow[r] & \cdots
\end{tikzcd}
\]
where the $n$-th arrow maps $u$ to $\beta\xi_n$.
\begin{proposition}\label{prop.AHSS}
If ${e^E} \in [\PP[-2],E]$ is  the Euler class of a complex orientation of $E$, then the above diagram induces a sequence of compatible isomorphisms
\[
\begin{tikzcd}[column sep=tiny]
\cdots \arrow[r] & {[\SSS,E][u]/(u^{n+2})} \arrow[r] \arrow[d, "\cong"'] & {[\SSS,E][u]/(u^{n+1})} \arrow[d, "\cong"'] \arrow[r] & {[\SSS,E][u]/(u^n)} \arrow[d, "\cong"'] \arrow[r] & \cdots \\
\cdots \arrow[r] & {[\mathbb{P}^{n+1}\C,E]} \arrow[r, "i_n^*"]                      & {[\mathbb{P}^n\C,E]} \arrow[r, "i_{n-1}^*"]                          & {[\mathbb{P}^{n-1}\C,E]} \arrow[r]                  & \cdots
\end{tikzcd}
\]
\end{proposition}
The proof of this proposition is just a consequence of the collapsing at the second page of the (cohomological) Atiyah--Hirzebruch spectral sequence for $\mathbb{P}^n\C$; a more detailed account of this argument can be found in \cite{adams} or \cite{kochman}. If we take the limit of the above diagram we get a commutative diagram 
\[
\begin{tikzcd}
{[\SSS,E][[u]]} \arrow[rr, "\cong", "u\,\mapsto \beta_E e^E"'] \arrow[dr, "u\,\mapsto \beta_E 1"'] && {[\PP,E]} \arrow[dl, "i_1^*"]
\\
& {[\mathbb{P}^1\C,E]}
\end{tikzcd}
\]
where the top horizontal arrow is an isomorphism, and in the left diagonal arrow $1$ denotes the unit of $[\SSS,E]$ seen as an element in $[\mathbb{P}^1\C[-2],E]$ via the isomorphism $\SSS\cong \mathbb{P}^1\C[-2]$. Conversely, the existence of such a commutative diagram is equivalent to the existence of a complex orientation of $E$. 
\begin{corollary}
Every even $2$-periodic ring spectrum is complex orientable.
\end{corollary}
\begin{proof}
Let $E$ be an even $2$-periodic ring spectrum. By the above discussion, to prove that $E$ is complex orientable we have to show that the distinguished ring morphism $[\SSS,E][u]\to [\mathbb{P}^1\C,E]$ sending $u \mapsto \beta_E 1$ admits a lift through $i_1^*$. The obstructions to such a lift are the obstructions to extending $\beta_E 1\colon \mathbb{P}^1\C\to E$ to a map $\beta e^E\colon \PP \to E$, through the sequence of skeleta inclusions
\[
\underbrace{\mathbb{P}^1\C}_{\text{2-skeleton}}=\underbrace{\mathbb{P}^1\C}_{
\text{3-skeleton}}\hookrightarrow \underbrace{\mathbb{P}^2\C}_{
\text{4-skeleton}}=\underbrace{\mathbb{P}^2\C}_{\text{5-skeleton}}\hookrightarrow\cdots
\]
so they lie in the singular cohomologies of $\PP$ with coefficients in the abelian groups $\pi^{\mathsf{Sp}}_{2n+1}(E)=[\SSS[2n+1],E]$. For an even cohomology theory these are all zero, so all the obstructions vanish.
\end{proof}

\section{Redde rationem villicationis tu\ae}\label{sec:redde}
Let $\rho_A$ and $\rho_B$ be two complex orientations for an even $2$-periodic $E_\infty$-ring spectrum $E$, and for every stably complex virtual vector bundle $V$ over $X$, denote by $\sigma_V^{A},\sigma_V^B$ the corresponding isomorphisms of $[X,E]$-modules $\sigma_V^{A},\sigma_V^{B}:[X,E]\xrightarrow{\cong}[X^V[-\rk{V}],E]$. If $f\colon X\to Y$ is a stably complex map, i.e., it is a smooth map such that $T_f$ is a stably complex virtual vector bundle over $X$, then we have two $E$-orientations on $-T_f$ and a commutative diagram
\[
\begin{tikzcd}
{[X,E]} \arrow[rr, "\cdot m_{-T_f}"] \arrow[rd, "\sigma_f^A"'] \arrow[rdd, "f_*^A"', bend right] &                                               & {[X,E]} \arrow[ld, "\sigma_f^B"] \arrow[ldd, "f_*^B", bend left] \\
                                                                                                 & {[X^{-T_f}[\dim X-\dim Y],E]} \arrow[d, "\varphi_{PT}^*"] &                                                                         \\
                                                                                                 & {[Y[\dim X-\dim Y],E]}                                  &                                                                        
\end{tikzcd}
\]
 In other words, we have
\[
f_*^A(a) = f_*^B (a\cdot m_{-Tf})
\]
for any $a \in [X,E]$.
We want to determine the multiplier $m_{-T_f}=f^*(m_{TY})\cdot m_{TX}^{-1}$. To do so, let $V$ be a stably complex vector bundle over $X$. 
As $m_{V+\mathbf{n}}=m_V$
for any $\mathbf{n}$, we may assume that $V$ is a complex vector bundle. Also, as the splitting principle works for every even 2-periodic $E_\infty$-ring spectrum $E$ as recalled above, it is not restrictive to assume that $V$ splits as a direct sum of complex line bundles $L_i$ classified by maps $\lambda_i\colon X\to \PP$. We then have 
\[
m_V=m_{L_1+\cdots L_k}=\prod_{i=1}^k m_{L_i} =\prod_{i=1}^k m_{\lambda_i^*L}=\prod_{i=1}^k \lambda_i^*m_{L},
\]
where $L=\mathcal{O}(1)$ is the universal line bundle on $\PP$. So
we just need to compute a single multiplier, namely $m_{L}$. 
This is defined by the equation $\tau^A_{L}=m_{L}\cdot \tau^B_{L}$. As the pullback along the zero section $\iota\colon \PP\to (\PP)^{L}$ is an isomorphism of $[\PP,E]$-modules, this is equivalent to the equation $e^A=m_{L}\cdot e^B$, where $e^A$ and $e^B$ are the Euler classes of the $E$-orientations $\rho_A$ and $\rho_B$, respectively.
By the conclusion of Section \ref{sec.7}, the Euler classes $e^A$ and $e^B$  uniquely determine a commutative diagram
\[
\begin{tikzcd}[column sep=tiny]
{[\SSS,E][[u]]} \arrow[rr, "u\,\mapsto \phi(u)"] \arrow[rd, "u\,\mapsto {\beta_E} e^A"'] &               & {[\SSS,E][[u]]} \arrow[ld, "u\,\mapsto {\beta_E} e^B"] \\
                                                     & {[\mathbb{P}^{\infty}\C,E]} 
\end{tikzcd}
\]
where every arrow is a ring isomorphism,
and so
\[
m_{L}=\frac{\phi(u)}{u}\biggr\vert_{u={\beta_E} e^B}.
\]
Therefore, 
for any complex vector bundle $V$ on $X$ we have
\[
m_V=\prod_{i} \frac{\phi(u)}{u}\biggr\vert_{u={\beta_E}\lambda_i^* e^B}.
\]
\begin{definition}\label{prodtodd}
It is customary to call $m_{-V}$ the \emph{Todd class} of $V$ relative to the two orientations $\rho_A$ and $\rho_B$ and to denote it $\mathrm{td}_{A,B}(V)$, i.e., 
\[
\mathrm{td}_{A,B}(V)=\prod_{i} \frac{u}{\phi(u)}\biggr\vert_{u={\beta_E}\lambda_i^*e^B}.
\]
By introducing the Todd function
\[
\mathrm{td}_{A,B}(u)=\frac{u}{\phi(u)},
\]
the above takes the form
\[
\mathrm{td}_{A,B}(V)=\prod_{i} \mathrm{td}_{A,B}(u)\biggr\vert_{u={\beta_E}\lambda_i^*e^B}.
\]
\end{definition}
\begin{remark}
Notice that, since the restrictions of $\beta_Ee^A$ and $\beta_Ee^B$ to $\mathbb{P}^1\C$ coincide (they are both equal to $\beta_E 1$), one has $\phi(u) = u + o(u)$, hence
\[
\mathrm{td}_{A,B}(u) = 1+o(1).
\]
\end{remark}
With this notation we have
\[
m_{-T_f}=\mathrm{td}_{A,B}(T_f)
\]
and the product formula for the Todd class from Definition \ref{prodtodd} tells how to express this as a product of characteristic classes for $TX$ and $TY$. Summing up, we have proven the following
\begin{theorem}[GHRR for a pair of complex orientations]
Let $\rho_A$ and $\rho_B$ be two complex orientations for an even $2$-periodic $E_\infty$-ring spectrum $E$, and let $f\colon X\to Y$ be a stably complex map. Then 
the two pushforward maps $f^A_*,f^B_*\colon [X,E]\to [Y[\dim X-\dim Y],E]$ are related by the Grothendieck--Hirzebruch--Riemann--Roch formula
\[
f_*^A(a) = f_*^B (a\cdot \mathrm{td}_{A,B}(T_f)),
\]
for any $a\in [X,E]$.
\end{theorem}
\begin{remark}
By multiplicativity and functoriality of the multipliers and by the projection formula, the above identity can be written as
\[
f_*^A(a)\cdot\mathrm{td}_{AB}(TY)=f_*^B(a\cdot\mathrm{td}_{AB}(TX)).
\]
\end{remark}

\section{Ipse nunc surgat nobis dicatus Mephistophilis!} 

Let now $\rho\colon MU\to E$ be a complex orientation of $E$ and let $\psi\colon E\to F$ be a morphism of homotopy commutative ring spectra. Then we have the following.
\begin{lemma}\label{lemma:HRR}
For any stably complex map $f\colon X \to Y$, the diagram
\[
\begin{tikzcd}
{[X,E]} \arrow[r, "\psi_*"] \arrow[d, "f_*^{\rho}"'] & {[X,F]} \arrow[d, "f_*^{\psi_*\rho}"] \\
{[Y[\dim X-\dim Y],E]} \arrow[r, "\psi_*"]                            & {[Y[\dim X-\dim Y],F]}                                
\end{tikzcd}
\]
commutes.
\end{lemma}
\begin{proof} The diagram
\[
\begin{tikzcd}
{[X,E]} \arrow[r, "\psi_*"] \arrow[d, "\sigma_{-T_f}"'] & {[X,F]} \arrow[d, "\psi_*\sigma_{-T_f}"] \\
{[X^{-T_f}[\dim X-\dim Y],E]} \arrow[r, "\psi_*"]                            & {[X^{-T_f}[\dim X-\dim Y],F]} \end{tikzcd}
\]
commutes, as $\psi_*$ maps the element $1$ in $[X,E]$ to the element 1 in $[X,F]$ and $\tau^{\psi_*\rho}_{-T_f}=\psi(\tau^\rho_{-T_f})$ by definition of $\psi_*\rho$. The diagram
\[
\begin{tikzcd}
{[X^{-T_f}[\dim X-\dim Y],E]} \arrow[r, "\psi_*"] \arrow[d, "\varphi_{PT}^*"'] & {[X^{-T_f}[\dim X-\dim Y],F} \arrow[d, "\varphi_{PT}^*"] \\
{[Y[\dim X-\dim Y],E]} \arrow[r, "\psi_*"]                            & {[Y[\dim X-\dim Y],F]} \end{tikzcd}
\]
trivially commutes as post-compositions and pre-compositions of morphisms commute.
\end{proof}
We also have
\begin{proposition}\label{prop:HRR}
Let $\rho_E\colon MU\to E$ and $\rho_F\colon MU\to F$ be complex orientations for the $E_\infty$-ring spectra $E$ and $F$, respectively, and let $\psi\colon E\to F$ be a morphism of ring spectra. Then, for any stably complex map $f\colon X\to Y$ and any $a\in [X,E]$, the following Grothendieck--Hirzebruch--Riemann--Roch like identity holds:
\[
\psi_*(f_*^{\rho_E}(a))=f_*^{\rho_F}(\psi_*(a)\cdot \mathrm{td}_{\psi_*\rho_E,\rho_F}(T_f))
\]
\end{proposition}
\begin{proof}
By Lemma \ref{lemma:HRR}, we have
\[
\psi_*(f_*^{\rho_E}(a))=f^{\psi_*\rho_E}_*(\psi_* (a)).
\]
The conclusion then follows from the second-last equation in Section \ref{sec:redde}.
\end{proof}
The statement of Proposition \ref{prop:HRR} is equivalent to the commutativity of the diagram
\[
\begin{tikzcd}
{[X,E]} \arrow[d, "f_*^{\rho_E}"'] \arrow[r, "{\mathrm{td}_{\psi_*\rho_E,\rho_F}(T_f)\cdot \psi_*}"] & {[X,F]} \arrow[d, "f_*^{\rho_F}"] \\
{[Y[\dim X-\dim Y],E]} \arrow[r, "\psi_*"]                                                           & {[Y[\dim X-\dim Y],F]}           
\end{tikzcd}
\]

\begin{remark}
As a particular case, one can consider $E=MU$ and $\rho_E$ to be the identity morphism of $MU$. Taking $F=HP_{\mathrm{ev}}\Q$ and $\rho_F\colon MU\to HP_{\mathrm{ev}}\Q$ the standard complex orientation $\rho_H$ of $HP_{\mathrm{ev}}\Q$, one sees from Proposition \ref{prop:HRR} that for any complex orientation $\psi\colon MU\to HP_{\mathrm{ev}}\Q$ and any complex manifold $X$ of complex dimension $n$ one has a commutative diagram

\[
\begin{tikzcd}
{[X,MU]} \arrow[d, "\pi_*^{\mathrm{id}_{MU}}"'] \arrow[rr, "{\mathrm{td}_{\psi,\rho_H}(TX)\cdot \psi_*}"] && {[X,HP_{\mathrm{ev}}\Q]} \arrow[d, "\pi_*^{\rho_H}"] \\
\Omega^U_{2n} \arrow[rr, "\psi_*"]                                                                        && {{\mathbb{Q}\beta^n},}                              
\end{tikzcd}
\]
where $\Omega^U_{2n}$ is the $2n$-dimensional complex cobordism group. The image of the unit element $1\in [X,MU]$ via $\pi_*^{\mathrm{id}_{MU}}$ is the complex cobordism class of $X$, see \cite{quillen1971elementary}. The pushforward map in even periodic rational singular cohomology induced by its standard complex orientation is the periodic version of the usual pushforward map in rational singular cohomology. In particular, if $X$ is a compact complex manifold, the pushforward map
\[
\int^{HP_{\mathrm{ev}}\Q}_X\colon \bigoplus_{i\in \mathbb{Z}}H^{2i}(X;\Q)\beta^i\to \Q\beta^{\dim_\C X}
\]
along the terminal morphism $\pi_X\colon X\to \ast$ is
\[
\int^{HP_{\mathrm{ev}}\Q}_X = \beta^{\dim_\C X}\int_X,
\]
where $\int_X$ is the usual integral in singular cohomology.. Finally, the morphism of abelian groups $\psi_\ast\colon \Omega^U_{2n}\to {\mathbb{Q}\beta^n}$ is the degree $-2n$ component of the Hirzebruch $\psi_\ast$-genus, i.e., of the graded rings homomorphism 
\[
\psi_\ast\colon \bigoplus_{n\in \mathbb{Z}}\Omega^U_{2n}\to \mathbb{Q}[\beta,\beta^{-1}].
\]
Therefore, as a particular case of Proposition \ref{prop:HRR} one finds Hirzebruch's genus formula:
\[
\psi_\ast([X])=\beta^{\dim_\C X}\int_X\mathrm{td}_{\psi,\rho_H}(TX).
\]
\end{remark}

\section{Ita ego ab hoc archetypo labor et decido} The archetype of the formula in Proposition \ref{prop:HRR} is obviously the classical Grothendieck--Hirzebruch--Riemann--Roch formula. Denote by $KU$ the spectrum representing complex $K$-theory and by $HP_{\mathrm{ev}}\Q$ the spectrum representing even periodic rational singular cohomology. Both spectra are multiplicative and even $2$-periodic.  With their standard complex orientations, their shifted Euler classes are given by $\beta_{K} e^{\rho_{K}}= \mathbf{1}_\C-L^{-1}\in [\PP,KU]$ and  by $\beta_H e^{\rho_H}=\beta c_1(L) \in [\PP,HP_{\mathrm{ev}}\Q]$, where $c_1$ is the first Chern class in singular cohomology and $\beta$ is a formal degree $-2$ variable; see, e.g. \cite[Example 1.1.5]{levine-morel} for the convention on the orientation of complex $K$-theory.

The Chern character ${{\rm ch}}\colon KU \to HP_{\mathrm{ev}}\Q$ provides a multiplicative map of even $2$-periodic cohomology theories (i.e. $\mathrm{ch}(\beta_{K})=\beta_H$, see Example \ref{example:chern-bott}) from complex $K$-theory to periodic rational singular cohomology. It can be seen as the composition
\[
\begin{tikzcd}
KU \arrow[r, "(-)_\Q"'] \arrow[rr, "{\rm ch}", bend left] & KU_\Q \arrow[r, "\cong"', "\Phi"] & HP_{\mathrm{ev}}\Q,
\end{tikzcd}
\]
where $(-)_\Q$ is the rationalization map and  $\Phi\colon KU_\Q \cong HP_{\mathrm{ev}}\Q$ is the equivalence given by the splitting of rational spectra in sums of Eilenberg--Mac Lane spectra, normalised so as to map the Bott element of $KU_\Q$ to the Bott element of $HP_{\mathrm{ev}}\Q$.
By the general splitting for rational spectra we have
\[
KU_\Q\cong \oplus_{i\in \mathbb{Z}} H\pi_i^{\mathsf{Sp}}(KU_\Q)[i]\cong \oplus_{i\in \mathbb{Z}} H(\pi_i^{\mathsf{Sp}} KU)\otimes \Q[i]\cong \oplus_{j\in \mathbb{Z}} H\Q[2j]=HP_{\mathrm{ev}}\Q.
\]
Moreover, naturality of these equivalences with respect to the monoidal structure of spectra implies that $KU_\Q\cong HP_{\mathrm{ev}}\Q$ is an equivalence if $E_\infty$-ring spectra.
A good reference for this result in a cohomological flavour can be found in \cite{hilton1971general}. Now, to get an explicit expression for the equivalence $\Phi\colon KU_\Q\cong HP_{\mathrm{ev}}\Q$ induced by rationalization, recall that $KU$ is generated by the class of the universal line bundle $L$, so we only need to determine $\Phi(L)$. This is an element in $[\PP,HP_{\mathrm{ev}}\Q]$, so by the results in Section \ref{sec.7}, there exists a unique formal power series 
\[
f(u)=\sum_{k=0}^\infty f_k u^k
\]
with coefficients in $\Q=[\SSS,HP_{\mathrm{ev}}\Q]$ such that 
\[
\Phi(L)=f(\beta c_1(L)).
\]
By naturality with respect to pullbacks, and since $\Phi$ is a ring homomorphism, we have
\[
1=\Phi(\mathbf{1}_\mathbb{C})=f_0,
\]
so we can write
\[
f(u)=e^{g(u)}
\]
for a unique formal power series
\[
g(u)=\sum_{k=0}^\infty g_k u^k
\]
with $g_0=0$. Again by naturality with respect to pullbacks and since the tensor product of vector bundles induces the product in $K$-theory, we have, for any $n\in \mathbb{Z}$,
\[
e^{g(n\, u)}\biggr\vert_{\beta c_1(L)}=e^{g(u)}\biggr\vert_{n\beta c_1(L)}=\Phi(L^{\otimes n})=\Phi(L)^{n}=e^{n\, g(u)}\biggr\vert_{\beta c_1(L)}.
\]
As evaluating at $\beta c_1(L)$ is an isomorphism $\Q[[u]]\xrightarrow{\sim} HP_{\mathrm{ev}}\Q(\PP)$, this gives
\[
g(n\, u)=n\ g(u),\qquad \text{for any }n\in \mathbb{Z},
\]
so $g(u)$ is a linear function: $g(u)=g_1\, u$ for some $g_1\in \Q$. Imposing that $\Phi$ preserves Bott elements we find
\[
\beta c_1(L\vert_{\mathbb{P}^1\C})=
1-e^{-g_1 \beta c_1(L\vert_{\mathbb{P}^1\C})}= g_1 \beta c_1(L\vert_{\mathbb{P}^1\C})),
\]
and so $g_1=1$, i.e., $f(u)=e^u$ and $\Phi=\mathrm{ch}$.
 As we are interested in
\[
\mathrm{td}_{\mathrm{ch}_*\rho_{K},\rho_{H}}(u),
\]
we have to identify the map $\phi$ determined by the commutative diagram
\[
\begin{tikzcd}[column sep=tiny]
{[\SSS,HP_{\mathrm{ev}}\Q][[u]]} \arrow[rr, "u\,\mapsto \phi(u)"] \arrow[rd, "u\,\mapsto \mathrm{ch}(\mathbf{1}_\C-L^{-1})"'] &               & {[\SSS,HP_{\mathrm{ev}}\Q][[u]]} \arrow[ld, "u\,\mapsto \beta c_1(L)"] \\
                                                     & {[\mathbb{P}^{\infty}\C,HP_{\mathrm{ev}}\Q]} &
\end{tikzcd}
\]
As $\mathrm{ch}(\mathbf{1}_\C-L^{-1})=1 - e^{-\beta c_1 (L)}$, the map $\phi$ is 
\[
\phi(u)=1-e^{-u}.
\]
Therefore the Todd function associated with the Chern character and the standard orientations of complex $K$-theory and even periodic rational cohomology is
\[
\mathrm{td}_{\mathrm{ch}_*\rho_{K},\rho_{H}}(u)=\frac{u}{1-e^{-u}}.
\]

\section{Exeunt}
Readers so patient to have read until here may be wondering what is the non-topological half of the Grothendieck--Hirzebruch--Riemann--Roch theorem we were hinting at in the title. To explain this, consider a compact complex manifold $X$ together with a \emph{holomorphic} vector bundle $V$ over it. Then the Hirzebruch--Riemann--Roch theorem can be stated as the identity
\[
\chi(X;V)=\int_X \mathrm{ch}(V)\, \mathrm{td}(TX),
\]
where on the left we have the holomorphic Euler characteristic of $X$ with coefficients in the holomorphic bundle $V$ (or, equivalently, in its sheaf $\mathcal{V}$ of holomorphic sections). This Euler characteristic is the pushforward in the analytical $K$-theory of $X$ (or equivalently in the bounded derived category of coherent sheaves on $X$) of the element $[V]$ in $KU^{\mathrm{an}}(X)$ to an element in $KU^{\mathrm{an}}(\ast)=\mathbb{Z}$ along the terminal morphism $\pi_X\colon X\to *$. That is, we have 
\[
\chi(X;V)=\pi_{X,*}^{KU,\mathrm{an}}([V]).
\]
By the discussion in the previous sections, the statement of the Hirzebruch--Riemann--Roch theorem can be rewritten as
\[
\overbrace{\pi_{X,*}^{KU,\mathrm{an}}([V])=\phantom{\pi_{X,*}^{KU}([V])}}^{\text{analytical/algebro-geometrical  part}} \hskip -44 pt \underbrace{\pi_{X,*}^{KU}([V]) = \int_X \mathrm{ch}(V)\, \mathrm{td}(TX)}_{\text{topological part}},
\]
where the right part of the identity is what we have discussed in this note, while the left part of the identity, i.e., the identification between the push-forward in analytic (or algebro-geometric) $K$-theory and the pushforward in topological $K$-theory is a deep result in the holomorphic (or algebro-geometric) setting, unattainable by purely topological methods. Analogous considerations apply to the more general case of the pushforward along a proper holomorphic map between holomorphic manifolds considered in the Grothendieck--Hirzebruch--Riemann--Roch theorem.

As a conclusion, let us recall how to determine the Todd function $u/(1-e^{-u})$ just by assuming the identity $\pi_{X,*}^{KU,\mathrm{an}}([V])=\pi_{X,*}^{KU}([V])$, where $V$ is a 
holomorphic vector bundle $V$ over a compact complex manifold $X$, holds for some complex orientation $\rho\colon MU\to KU$ of the topological complex $K$-theory. Under these assumptions, by the discussion in the previous sections we will have in particular the identities  
\[
\chi(\mathbb{P}^n\C;\mathcal{O}_{\mathbb{P}^n\C})=\int_{\mathbb{P}^n\C} \mathrm{td}_{\mathrm{ch}_*\rho,\rho_H}(T{\mathbb{P}^n\C})
=\beta^{-n}
\int_{\mathbb{P}^n\C}^{HP_{\mathrm{ev}}\Q} \mathrm{td}_{\mathrm{ch}_*\rho,\rho_H}(T{\mathbb{P}^n\C})
\]
for a suitable formal power series $\mathrm{td}_{\mathrm{ch}_*\rho,\rho_H}(u)$. By the Euler exact sequence 
\[
0\to \mathcal{O}_{\mathbb{P}^n\C}\to \bigoplus_{i=0}^n\mathcal{O}_{\mathbb{P}^n\C}(1)\to T\mathbb{P}^n\C\to 0
\]
we get
\[
\mathrm{td}_{\mathrm{ch}_*\rho,\rho_H}(T{\mathbb{P}^n\C})=\mathrm{td}_{\mathrm{ch}_*\rho_K,\rho_H}(\beta c_1(\mathcal{O}_{\mathbb{P}^n\C}(1)))^{n+1}
\]
so that, writing
\[
(\mathrm{td}_{\mathrm{ch}_*\rho,\rho_H}(u))^{n+1}=\sum_{k=0}^\infty a_{n+1,k} u^k
\]
we obtain 
\[
\chi(\mathbb{P}^n\C;\mathcal{O}_{\mathbb{P}^n\C})=a_{n+1,n}.
\]
On the other hand
\[
\chi(\mathbb{P}^n\C;\mathcal{O}_{\mathbb{P}^n\C})=\sum_{q=0}^{n}h^{0,q}(\mathbb{P}^n\C)=1.
\]
Therefore the formal series $\mathrm{td}_{\mathrm{ch}_*\rho,\rho_H}(u)$ must satisfy $a_{n+1,n}=1$ for every $n$, i.e., it must have the property that the coefficient of $u^n$ in $(\mathrm{td}_\rho(u))^{n+1}$ equals 1 for any $n$. A classical computation using the Lagrange inversion formula shows that there exists a unique formal power series with this property: the power series expansion of
\[
\mathrm{td}_{\mathrm{ch}_*\rho_K,\rho_H}(u)=\frac{u}{1-e^{-u}}.
\]
As the Chern character is an isomorphism from rationalized complex $K$-theory to even 2-periodic rational singular cohomology, this shows that the only possible complex orientation of topological $K$-theory for which one can have the complex analytic/algebro-geometric half of the Grothendieck--Hirzebruch--Riemann--Roch theorem is $\rho_K$, i.e., the one with (shifted) Euler class $\beta_{KU}e^{KU}=\mathbf{1}_\C-L^{-1}$, thus motivating this apparently less natural choice with respect to $\beta_{KU}e^{KU} \allowbreak =L-\mathbf{1}_\C$. Clearly, as far as one is not concerned with the complex analytic/\allowbreak algebro-geometric half of the theorem, this second orientation, with corresponding Todd function $u/(e^u-1)$, is an equally valid choice, and it is actually quite a common choice for defining a complex orientation of topological $K$-theory in algebraic topology.

\nocite{*}
\bibliographystyle{alpha}
\bibliography{bibliography}
\vfill
\tiny \noindent 
The reproduction of part of a letter by Alexander Grothendieck has been extracted from \\
{\tt https://commons.wikimedia.org/wiki/File:Grothendieck-Riemann-Roch.jpg}\\ where it is licensed as ``This work is ineligible for copyright and therefore in the public domain because it consists entirely of information that is common property and contains no original authorship.''
\vskip .6 cm
\end{document}